\documentclass[12pt]{preprintCL}

\usepackage[utf8]{inputenc}

\usepackage[T1]{fontenc}
\usepackage[english]{babel}
\usepackage{amsmath}
\usepackage{amsfonts}
\usepackage{amssymb}
\usepackage{amsthm}
\usepackage{graphicx}
\usepackage{array}
\usepackage{color}
\usepackage{mathrsfs}
\usepackage{hyperref}
\usepackage{esint} 
\usepackage{tikz}
\usepackage{upgreek}
\usepackage{enumitem}

\definecolor{darkgreen}{rgb}{0.1,0.7,0.1}

\newcommand{\BLUE}[1]{#1}

\newtheorem{lemma}{Lemma}[section]
\newtheorem{theorem}{Theorem}
\newtheorem{proposition}[lemma]{Proposition}

\newtheorem{remark}[lemma]{Remark}
\newtheorem{definition}[lemma]{Definition}

\newcommand{\bn}[1]{{[\kern-0.5ex] #1 
    [\kern-0.5ex]}}

\newcommand{\eps}{\varepsilon}

\newcommand\symb[2][\bf]{{\mathchoice{\hbox{#1#2}}{\hbox{#1#2}}%
        {\hbox{\scriptsize#1#2}}{\hbox{\tiny#1#2}}}}

\def\R{{\symb R}}
\def\N{{\symb N}}
\def\Z{{\symb Z}}
\def\C{{\symb C}}

\def\P{{\symb P}}

\def\un{\mathbf{1}}

\def\rm{\mathrm{m}}

\def\${|\!|\!|}

\newcommand{\crochet}[1]{\left\langle #1 \right\rangle}
\newcommand{\norm}[1]{\left\| #1 \right\|}
\newcommand{\normm}[1]{{\left\vert\!\left\vert\!\left\vert #1 \right\vert\!\right\vert\!\right\vert}}

\renewcommand{\P}{\operatorname{\mathbb{P}}}
\newcommand{\E}{\operatorname{\mathbb{E}}}

\renewcommand{\Re}{\operatorname{Re}}
\newcommand{\dd}{\operatorname{d\!}}

\newcommand{\supp}{\mathtt{supp}}

\newcommand{\bbC}{\mathbb{C}}

\newcommand{\bbZ}{\mathbb{Z}}

\newcommand{\cA}{\mathcal{A}}
\newcommand{\cB}{\mathcal{B}}
\newcommand{\cC}{\mathcal{C}}
\newcommand{\cD}{\mathcal{D}}
\newcommand{\cE}{\mathcal{E}}
\newcommand{\cF}{\mathcal{F}}
\newcommand{\cG}{\mathcal{G}}
\newcommand{\cH}{\mathcal{H}}
\newcommand{\cI}{\mathcal{I}}

\newcommand{\cM}{\mathcal{M}}

\newcommand{\cP}{\mathcal{P}}

\newcommand{\cR}{\mathcal{R}}

\newcommand{\cT}{\mathcal{T}}
\newcommand{\cU}{\mathcal{U}}
\newcommand{\cV}{\mathcal{V}}

\newcommand{\ccD}{\mathscr{D}}
\newcommand{\ccE}{\mathscr{E}}

\newcommand{\ccH}{\mathscr{H}}

\newcommand{\ccM}{\mathscr{M}}

\newcommand{\ccR}{\mathscr{R}}

\hypersetup{
citecolor=blue,
colorlinks=true, 
breaklinks=true, 
urlcolor= blue, 
linkcolor= black, 
bookmarksopen=true, 
pdftitle={Anderson Full Space}, 
}

\usetikzlibrary{snakes}
\usetikzlibrary{decorations}
\usetikzlibrary{decorations.markings}
\usetikzlibrary{positioning}
\usetikzlibrary{shapes}
\usetikzlibrary{arrows}
\usetikzlibrary{arrows.meta} 

\colorlet{symbols}{black!50}
\colorlet{testcolor}{green!60!black}
\definecolor{connection}{rgb}{0.7,0.1,0.1}
\definecolor{lblue}{rgb}{0.1,0.5,1}

\makeatletter
\pgfdeclareshape{crosscircle}
{
	\inheritsavedanchors[from=circle] 
	\inheritanchorborder[from=circle]
	\inheritanchor[from=circle]{north}
	\inheritanchor[from=circle]{north west}
	\inheritanchor[from=circle]{north east}
	\inheritanchor[from=circle]{center}
	\inheritanchor[from=circle]{west}
	\inheritanchor[from=circle]{east}
	\inheritanchor[from=circle]{mid}
	\inheritanchor[from=circle]{mid west}
	\inheritanchor[from=circle]{mid east}
	\inheritanchor[from=circle]{base}
	\inheritanchor[from=circle]{base west}
	\inheritanchor[from=circle]{base east}
	\inheritanchor[from=circle]{south}
	\inheritanchor[from=circle]{south west}
	\inheritanchor[from=circle]{south east}
	\inheritbackgroundpath[from=circle]
	\foregroundpath{
		\centerpoint%
		\pgf@xc=\pgf@x%
		\pgf@yc=\pgf@y%
		\pgfutil@tempdima=\radius%
		\pgfmathsetlength{\pgf@xb}{\pgfkeysvalueof{/pgf/outer xsep}}%
		\pgfmathsetlength{\pgf@yb}{\pgfkeysvalueof{/pgf/outer ysep}}%
		\ifdim\pgf@xb<\pgf@yb%
		\advance\pgfutil@tempdima by-\pgf@yb%
		\else%
		\advance\pgfutil@tempdima by-\pgf@xb%
		\fi%
		\pgfpathmoveto{\pgfpointadd{\pgfqpoint{\pgf@xc}{\pgf@yc}}{\pgfqpoint{-0.707107\pgfutil@tempdima}{-0.707107\pgfutil@tempdima}}}
		\pgfpathlineto{\pgfpointadd{\pgfqpoint{\pgf@xc}{\pgf@yc}}{\pgfqpoint{0.707107\pgfutil@tempdima}{0.707107\pgfutil@tempdima}}}
		\pgfpathmoveto{\pgfpointadd{\pgfqpoint{\pgf@xc}{\pgf@yc}}{\pgfqpoint{-0.707107\pgfutil@tempdima}{0.707107\pgfutil@tempdima}}}
		\pgfpathlineto{\pgfpointadd{\pgfqpoint{\pgf@xc}{\pgf@yc}}{\pgfqpoint{0.707107\pgfutil@tempdima}{-0.707107\pgfutil@tempdima}}}
	}
}

\pgfdeclareshape{cross2}
{
	\inheritsavedanchors[from=rectangle]
	\inheritanchorborder[from=rectangle]
	\inheritanchor[from=rectangle]{center}
	\inheritanchor[from=rectangle]{north}
	\inheritanchor[from=rectangle]{south}
	\inheritanchor[from=rectangle]{east}
	\inheritanchor[from=rectangle]{west}
	\inheritanchor[from=rectangle]{north east}
	\inheritanchor[from=rectangle]{north west}
	\inheritanchor[from=rectangle]{south east}
	\inheritanchor[from=rectangle]{south west}
	
	\inheritbackgroundpath[from=rectangle]
	
	\behindforegroundpath{%
		\pgfextractx{\pgf@xa}{\southwest}%
		\pgfextracty{\pgf@ya}{\southwest}%
		\pgfextractx{\pgf@xb}{\northeast}%
		\pgfextracty{\pgf@yb}{\northeast}%
		
		\pgfsetlinewidth{0.5pt} 
		
		\pgfpathmoveto{\pgfpoint{0.95\pgf@xa}{0.95\pgf@ya}}%
		\pgfpathlineto{\pgfpoint{0.95\pgf@xb}{0.95\pgf@yb}}%
		
		\pgfpathmoveto{\pgfpoint{0.95\pgf@xa}{0.95\pgf@yb}}%
		\pgfpathlineto{\pgfpoint{0.95\pgf@xb}{0.95\pgf@ya}}%
		
		\pgfusepath{stroke}
	}
}
\makeatother

\tikzset{
	xi/.style={very thin,circle,fill=white,draw=black,inner sep=0pt,minimum size=1.1mm},
	zeta1/.style={very thin,circle,fill=black!18,draw=black,inner sep=0pt,minimum size=1.1mm},
	zeta2/.style={very thin,circle,fill=black,draw=black,inner sep=0pt,minimum size=1.1mm},
	eta/.style={thin,rectangle,fill=red!30,draw=red,inner sep=0pt,minimum size=1.1mm},
	aeta/.style={very thin,rectangle,fill=white,draw=black,inner sep=0pt,minimum size=1.1mm},
	etab/.style={thin,rectangle,fill=red!30,draw=red,inner sep=0pt,minimum size=1.6mm},
	etabx/.style={cross2,fill=red!30,draw=red,inner sep=0pt,minimum size=1.6mm},
	aetabx/.style={cross2,fill=white,draw=black,inner sep=0pt,minimum size=1.6mm},
	aetab/.style={very thin,rectangle,fill=white,draw=black,inner sep=0pt,minimum size=1.6mm},
	xix/.style={crosscircle,fill=white,draw=black,inner sep=0pt,minimum size=1.2mm},
	xib/.style={very thin,circle,fill=white,draw=black,inner sep=0pt,minimum size=1.6mm},
	zeta2b/.style={very thin,circle,fill=black,draw=black,inner sep=0pt,minimum size=1.6mm},	
	zeta1b/.style={very thin,circle,fill=black!18,draw=black,inner sep=0pt,minimum size=1.6mm},	
	xibx/.style={crosscircle,fill=white,draw=black,inner sep=0pt,minimum size=1.6mm},
	not/.style={thin,circle,fill=black,draw=black,inner sep=0pt,minimum size=0.3mm},
	>=stealth,
	root/.style={circle,fill=testcolor,inner sep=0pt, minimum size=2mm},
	dot/.style={circle,fill=black,inner sep=0pt, minimum size=1mm},
	var/.style={circle,fill=black!10,draw=black,inner sep=0pt, minimum size=2mm},
	dotred/.style={circle,fill=black!50,inner sep=0pt, minimum size=2mm},
	generic/.style={semithick,shorten >=1pt,shorten <=1pt},
	dist/.style={ultra thick,draw=testcolor,shorten >=1pt,shorten <=1pt},
	testfcn/.style={ultra thick,testcolor,shorten >=1pt,shorten <=1pt,->},
	G/.style={semithick,shorten >=1pt,shorten <=1pt,black,->},
	Gdiff/.style={dashed,semithick,shorten >=1pt,shorten <=1pt,black,->},
	rho/.style={dotted,semithick,shorten >=1pt,shorten <=1pt},
	renorm/.style={shape=circle,fill=white,inner sep=1pt},
	labl/.style={shape=rectangle,fill=white,inner sep=1pt},
}

\makeatletter
\def\DeclareSymbol#1#2#3{\expandafter\gdef\csname MH@symb@#1\endcsname{\tikz[baseline=#2,scale=0.15,draw=symbols]{#3}}\expandafter\gdef\csname MH@symb@#1s\endcsname{\scalebox{0.7}{\tikz[baseline=#2,scale=0.15,draw=symbols]{#3}}}}
\def\<#1>{\csname MH@symb@#1\endcsname}
\makeatother

\DeclareSymbol{Xi22}{0.5}{\draw (0,0) node[xi] {} -- (-1,1) node[not] {} -- (0,2) node[xi] {};}

\DeclareSymbol{Xi2}{-2}{\draw (0,-0.25) node[xi] {} -- (-1,1) node[xi] {};}
\DeclareSymbol{Xi3}{0}{\draw (0,0) node[xi] {} -- (-1,1) node[xi] {} -- (0,2) node[xi] {};}
\DeclareSymbol{Xi4}{2}{\draw (0,0) node[xi] {} -- (-1,1) node[xi] {} -- (0,2) node[xi] {} -- (-1,3) node[xi] {};}
\DeclareSymbol{Xi2X}{-2}{\draw (0,-0.25) node[xi] {} -- (-1,1) node[xix] {};}
\DeclareSymbol{XXi2}{-2}{\draw (0,-0.25) node[xix] {} -- (-1,1) node[xi] {};}

\DeclareSymbol{IXi2}{0}{\draw (0,-0.25) node[not] {} -- (-1,1) node[xi] {} -- (0,2) node[xi] {};}
\DeclareSymbol{IXi^2}{-1}{\draw (-1,1) node[xi] {} -- (0,0) node[not] {} -- (1,1) node[xi] {};}

\DeclareSymbol{XiX}{-2.8}{\node[xibx] {};}
\DeclareSymbol{Xi}{-2.8}{\node[xib] {};}
\DeclareSymbol{IXiX}{-1}{\draw (0,-0.25) node[not] {} -- (-1,1) node[xix] {};}

\DeclareSymbol{Xi3b}{-1}{\draw (-1,1) node[xi] {} -- (0,0) node[xi] {} -- (1,1) node[xi] {};}

\DeclareSymbol{IXi3}{2}{\draw (0,-0.25) node[not] {} -- (-1,1) node[xi] {} -- (0,2) node[xi] {} -- (-1,3) node[xi] {};}
\DeclareSymbol{IXi}{-2}{\draw (0,-0.25) node[not] {} -- (-1,1) node[xi] {};}
\DeclareSymbol{XiI}{-2}{\draw (0,-0.25) node[xi] {} -- (-1,1) node[not] {};}

\DeclareSymbol{Xi4b}{-1}{\draw(0,1.5) node[xi] {} -- (0,0); \draw (-1,1) node[xi] {} -- (0,0) node[xi] {} -- (1,1) node[xi] {};}
\DeclareSymbol{Xi4b'}{-1}{\draw(0,1.5) node[xi] {} -- (0,-0.2); \draw (-1,1) node[xi] {} -- (0,-0.2) node[not] {} -- (1,1) node[xi] {};}
\DeclareSymbol{Xi4c}{0}{\draw (0,1) -- (0.8,2.2) node[xi] {};\draw (0,-0.25) node[xi] {} -- (0,1) node[xi] {} -- (-0.8,2.2) node[xi] {};}
\DeclareSymbol{Xi4d}{-4.5}{\draw (0,-1.5) node[not] {} -- (0,0); \draw (-1,1) node[xi] {} -- (0,0) node[xi] {} -- (1,1) node[xi] {};}
\DeclareSymbol{Xi4e}{0}{\draw (0,2) node[xi] {} -- (-1,1) node[xi] {} -- (0,0) node[xi] {} -- (1,1) node[xi] {};}
\DeclareSymbol{Xi4e'}{0}{\draw (0,2) node[xi] {} -- (-1,1) node[xi] {} -- (0,-0.2) node[not] {} -- (1,1) node[xi] {};}

\begin{document}

\title{Construction and spectrum of the Anderson Hamiltonian with white noise potential on $\R^2$ and $\R^3$}
\author{Yueh-Sheng Hsu\footnote{Universit\'e Paris-Dauphine, PSL University, CNRS, UMR 7534, CEREMADE, 75016 Paris, France. \email{hsu@ceremade.dauphine.fr}}\; and Cyril Labb\'e\footnote{Universit\'e Paris Cit\'e, Laboratoire de Probabilit\'es, Statistique et Mod\'elisation, UMR 8001, F-75205 Paris, France. \email{clabbe@lpsm.paris}}}

\vspace{2mm}

\date{\today}
\maketitle

\begin{abstract}
	
	We propose a simple construction of the Anderson Hamiltonian with white noise potential on $\R^2$ and $\R^3$ based on the solution theory of the parabolic Anderson model. It relies on a theorem of Klein and Landau~\cite{KL81} that associates a unique self-adjoint generator to a symmetric semigroup satisfying some mild assumptions. Then, we show that almost surely the spectrum of this random Schr\"odinger operator is $\R$. To prove this result, we extend the method of Kotani~\cite{Kot85} to our setting of singular random operators.
	
	\medskip
	
\end{abstract}

\setcounter{tocdepth}{1}
\tableofcontents

\section{Introduction}\label{sec:intro}

This article establishes the rigorous construction and some spectral properties of the random Schr\"odinger operator
$$ \cH := -\Delta + \xi\;,\quad\mbox{ on }\R^d\;,$$
where $d\in\{2,3\}$, $\Delta$ is the continuous Laplacian and $\xi$ is a white noise on $\R^d$.\\

Let us mention that the spectral properties of random Schr\"odinger operators $-\Delta + V$ have received much attention in the mathematical physics community, most notably in the discrete case where $\Delta$ is the Laplacian on $\Z^d$ and $V(k),k\in\Z^d$ is a collection of i.i.d.~r.v. In particular, the Anderson Localization phenomenon~\cite{Anderson58} has given rise to a large literature. The potential considered in the present article arises as a scaling limit of a large class of discrete or continuous random fields, thus motivating the study of the operator $\cH$. From a physical perspective, this operator can be seen as an idealized model in a situation where the potential $V$ has very small correlation length.\\

Since white noise is distribution-valued, the mere definition of the operator $\cH$ is delicate. Actually, the operator requires renormalisation by infinite constants and recent breakthroughs~\cite{GIP,Hai14} in the field of stochastic PDE provide the required tools to carry out this task. Let us give a brief description of the renormalisation procedure. One considers the smoothed out noise $\xi_\eps := \xi * \varrho_\eps$ where $\varrho_\eps(x) := \eps^{-d} \varrho(x/\eps)$, $x\in\R^d$, for some even, smooth function $\varrho$ compactly supported in the unit ball of $\R^d$ and integrating to $1$. If one replaces $\xi$ by $\xi_\eps$ then the definition of the operator falls into the scope of classical results~\cite{Kat72,FL74}. However the sequence of operators does not converge as $\eps\downarrow 0$. Instead, one needs to identify carefully a diverging (with $\eps$) constant $C_\eps$ such that the collection of self-adjoint operators
$$ \cH_\eps := -\Delta + \xi_\eps + C_\eps\;,$$
converges to some limit that we call $\cH$. Let us mention that $C_\eps$ diverges as $(2\pi)^{-1} \log \eps^{-1}$ in dimension $2$, while in dimension $3$ it diverges as $c_\varrho \eps^{-1} + \ln \eps^{-1}$ for some constant $c_\varrho>0$.\\

This programme has been carried out in finite volume (a torus, a bounded box) in dimensions $2$ and $3$ by several authors~\cite{AC15,GUZ,Lab19,Cv21,Mou21,MvZ,BDM}. The convergence of $\cH_\eps$ towards $\cH$ then holds in the norm resolvent sense, in probability. Let us point out that in finite volume, the operator is bounded from below so that its resolvent set admits some real values and fixed point argument can be applied to build the associated resolvent operators.\\
It turns out that in infinite volume the construction is much harder. The main reason is that the operator is no longer bounded from below (see for instance the large volume asymptotics on the ground state~\cite{Cv21,HL22}) so that the above stragegy of proof is no longer applicable.\\
In a recent work, Ugurcan~\cite{Ugu22} proposed a construction of $\cH$ on $\R^2$ by adapting a celebrated criterion~\cite{FL74} of Faris and Lavine ensuring essential self-adjointness for Schr\"odinger operators. Let us however point out that Ugurcan does not address the convergence of $\cH_\eps$ towards $\cH$. More recently, Ueki~\cite{Uek23} proposed another construction of $\cH$ on $\R^2$, based on the heat semigroup approach of~\cite{Mou21}: therein, it is proven that $\cH_\eps$ converges in the strong resolvent sense towards $\cH$ and that the spectrum of $\cH$ is almost surely equal to $\R$.\\
On the other hand, the construction of $\cH$ on $\R^3$ has not been addressed so far.\\

The two \BLUE{aforementioned} works manage to identify the domain of the random operator $\cH$ on $\R^2$ using the paracontrolled calculus. Let us point out that these works are technically quite involved. In the present article, we present a relatively simple construction of $\cH$, in both dimensions $2$ and $3$, which is based on the solution theory of the parabolic Anderson model (PAM)
\begin{equation}\label{Eq:PAM}
	\begin{cases}
		\partial_t u = \Delta u - u\xi\;,\quad \mbox{ on }\R^d\;,\\
		u(t=0,\cdot) = f(\cdot)\;.
	\end{cases}
\end{equation}
The construction of this PDE is a priori less delicate than the construction of the self-adjoint operator $\cH$: indeed, solution theory for PDE is flexible as one simply has to identify a (reasonable) space in which existence and uniqueness holds, while self-adjointness is a delicate property that requires to identify a (random) dense subset of $L^2$ on which the operator acts. Let us point out that the construction of the (PAM) was established in~\cite{HL15,HL18} respectively on $\R^2$ and $\R^3$.\\
We define the operator $\cH$ as the generator of the semigroup associated to the collection of solutions to this PDE. This is the content of our first result:

\begin{theorem}\label{Th:Construction}
	In dimensions $2$ and $3$, there exists a random operator $\cH$ which is self-adjoint on (a dense subset of) $L^2(\R^d,dx)$ and which is the limit in probability of $\cH_{\epsilon}$ as $\epsilon\downarrow 0$ in the strong resolvent sense. For any $t\ge 0$, the domain of the operator $e^{-t\cH}$ contains all functions $f\in L^2(\R^d,dx)$ with compact support and $e^{-t\cH} f$ coincides with the solution of \eqref{Eq:PAM} at time $t$, starting from $f$ at time $0$.
\end{theorem}

The proof of Theorem \ref{Th:Construction} is split into three steps:

\begin{tabular}{ccccccc}
	\mbox{Noise} &$\overset{(1)}{\longrightarrow}$ &\mbox{Enhanced noise} &$\overset{(2)}{\longrightarrow}$ &\mbox{ (PAM) }& $\overset{(3)}{\longrightarrow}$ &\mbox{Generator}\\
	$\xi$ & &$Q(\xi)$ && $u$ && $\cH$
\end{tabular}\\

Step $(1)$ is standard in the field of singular stochastic PDEs: it associates in a measurable (but typically non-continuous) way an enhanced noise to a (typically rough) noise. The point is that the additional data contained in the enhanced noise will allow to recover continuity in the subsequent steps. Note that the renormalisation is implemented in this step.\\

Step $(2)$ is deterministic: given the enhanced noise $Q(\xi)$, one defines a solution theory to the (PAM). In dimension $2$ it relies on relatively elementary arguments that we recall in Section \ref{sec:2d}. On the other hand, the construction of (PAM) in dimension $3$ is involved and relies on the theory of regularity structures~\cite{Hai14}: we will use the results from~\cite{HL18}, see Section \ref{sec:3d}.\\

If we denote by $u^f(t,\cdot)$ the solution of (PAM), then $P_t f := u^f(t,\cdot)$ is a semigroup. Step $(3)$ builds (in a deterministic way) a unique self-adjoint operator $\cH$ satisfying $P_t f = e^{-t\cH} f$. However, extracting $\cH$ from this semigroup is much more subtle than one may think.\\
A natural attempt would be to construct the resolvent of $\cH$ by integration in time of the semigroup: $\int_0^\infty e^{-at} u^f(t,\cdot) dt$. However the operator is not bounded from below so that nothing ensures that $(\cH + a)^{-1}f$ exists for real parameters $a$. In addition, the solution theory of (PAM) provides very bad a priori bounds on the growth in time of $u^f(t,\cdot)$ so that one cannot single out nice initial conditions $f$ for which the above integral in time converges. Actually these functions $f$ for which the integral in time converges should be functions that, informally speaking, ``avoid'' \BLUE{those} regions of space where $\xi$ is very large, thus confirming that one cannot identify a simple set of such functions.\\
A second approach would consist in applying Friedrich's extension\BLUE{~\cite[Thm 2.13]{Tes14}} to the quadratic forms $\langle f, u^f(t,\cdot)\rangle$ (for nice enough functions $f$) and in exploiting the semigroup property satisfied by the solution of (PAM): we did not manage to conclude with this approach. We then came \BLUE{across} a (beautiful) result of Klein and Landau~\cite{KL81} that provides the right framework. This is presented in Section \ref{sec:semigroup}, and the required estimates on the (PAM) to apply this result and deduce Theorem \ref{Th:Construction} are presented in Sections \ref{sec:2d} and \ref{sec:3d}.\\

Let us mention that the construction of the operator via its semigroup comes with nice continuity properties that allow, in particular, to prove the strong resolvent convergence of the statement, see Propositions \ref{prop:continuity_H} and \ref{prop:continuity_H_3d}. We would like to emphasise that the approach presented here is robust and can certainly be applied to a large class of singular random differential operators.\\

Our second result is the identification of the spectrum of $\cH$.

\begin{theorem}\label{Th:Spectrum}
	In dimensions $2$ and $3$, almost surely the spectrum of $\cH$ is $\R$.
\end{theorem}

The fact that the spectrum is almost surely a deterministic set is a consequence of the ergodicity of white noise combined with the following commutation property involving the shifted noise $\theta_x \xi(\cdot) = \xi(\cdot-x)$ and the translation operators $\cT_x f(\cdot) = f(\cdot+x)$ (see Section \ref{sec:spectrum} for more details)
$$ \cH(\theta_x \xi) = \cT_x^* \cH(\xi) \cT_x\;,\quad \forall x\in\R^d\;.$$
While this property is clear at an informal level, to establish it rigorously one needs to construct the enhanced noise associated to the shifted noise $\theta_x \xi$ simultaneously for all $x\in \R^d$. This is the purpose of Lemmas \ref{lem:Omega02d} and \ref{lem:Omega03d}.\\

A usual strategy to show that some value $\lambda \in \R$ belongs to the spectrum of a self-adjoint operator $A$ is to build a Weyl sequence, namely, a sequence of functions $f_n$ \BLUE{in the domain of $A$ and} of unit $L^2$-norm which are such that $(A-\lambda)f_n$ converges to $0$ in $L^2$.
This argument is implemented for the operator $\cH_\eps$ in Proposition \ref{prop:Sigma_eps}, and allows to show that the spectrum of the latter is the whole line $\R$ a.s. For the limiting operator $\cH$ however, this argument is delicate to implement since the domain of the operator is made of non-smooth functions that depend in a subtle way on the enhanced noise. Let us mention that Ueki~\cite{Uek23} managed to implement Weyl sequences in dimension $2$ using delicate approximation arguments, and then showed that the spectrum is the whole line $\R$ almost surely.\\
In the present work, we rely on a different argument that works both in dimensions $2$ and $3$. The underlying idea comes from a result of Kotani~\cite{Kot85} which can be spelled out in the following way: if one shows that the (topological) support of $\xi_\eps$ is included into the (topological) support of $\xi$, then the spectrum of $\cH_\eps$ is included in the spectrum of $\cH$. Since Proposition \ref{prop:Sigma_eps} shows that the former is $\R$, we deduce the asserted result.\\
While this informal idea is rigourous for function-valued potentials as shown in~\cite{Kot85}, in the singular context considered in this article it is only heuristic. Indeed, the argument of Kotani relies on the continuity of the operator w.r.t.~to the noise: in the present situation, the operator is only measurable w.r.t.~the noise but is continuous w.r.t.~the enhanced noise. Consequently, we need to show that the (topological) support of the \emph{enhanced noise} associated to $\xi_\eps$ is included into the (topological) support of the enhanced noise associated to $\xi$, and we will show in Section \ref{sec:spectrum} that we can extend the argument of Kotani to that setting.\\
In dimension $2$, we provide a complete proof of the inclusion of the supports of the enhanced noises, see Subsection \ref{subsec:supports2d}. In dimension $3$, we rely on a general result of Hairer and Sch\"onbauer~\cite{HS21} that establishes a support theorem for singular stochastic PDE, see Subsection \ref{subsec:supports3d}.

\subsubsection*{Notations}
Let us introduce some notations for the rest of the article. For some given parameters $a > 0$ and $\ell \in\R$, we introduce the weight functions
\begin{equation}\label{Eq:weights}
	p_a(x) = (1+|x|)^a\;,\quad e_\ell(x) = e^{\ell(1+|x|)}\;,\quad x\in\R^d\;.
\end{equation}
Let us mention that for any given $a>0$, there exists a constant $C>0$ such that for all $0\le s\le t$ and all $\ell\in\R$
\begin{equation}\label{Eq:TrickWeights}
	\sup_{x\in\R^d} \frac{p_a(x) e_{\ell+s}(x)}{e_{\ell+t}(x)} \le C (t-s)^{-a}\;.
\end{equation}
For $w=p_a$ or $w=e_\ell$, we let $L^2_w$ be the weighted version of the usual $L^2$ space: its norm is defined through
$$ \|f\|_{L^2_w}^2 := \int_{x\in\R^d} \frac{\vert f(x)\vert^2}{w(x)} dx\;.$$
For $\alpha < 0$, we let $\cC^{\alpha}_w$ be the closure of all compactly supported, smooth functions under the norm
\[
\| f\|_{\cC^{\alpha}_w} := \sup_{x\in \R^d} \sup_{\lambda \in (0,1]} \sup_{\varphi \in \cB_{\lceil -\alpha \rceil}} \frac{\vert \langle f,\varphi^\lambda_x\rangle \vert }{w(x)}\;,
\]
where $\varphi^\lambda_x(y) := \lambda^{-d} \varphi((y-x)/\lambda)$ and $\cB_r$ is the set of all functions, supported in the unit ball of $\R^d$, whose $C^r$-norm is bounded by $1$. Similarly we let $\cH^{\alpha}_w$ be the weighted Besov space with parameters $p=q=2$ (or equivalently, the weighted Sobolev space) of regularity index $\alpha$.\\
We point out a fact that will be used many times in this article. Let $\chi:\R^d\to [0,1]$ be a compactly supported smooth function such that $\sum_{k\in\Z^d} \chi(\cdot-k) =1$. Then for any given weight $w$, there exists a constant $C>0$
\begin{equation}\label{Eq:Calphaw}
	\| f\|_{\cC^{\alpha}_w} \le C \sup_{k\in\Z^d} \frac{\| f(\cdot)  \chi(\cdot-k)\|_{\cC^{\alpha}}}{w(k)}\;.
\end{equation}

\paragraph{Acknowledgements.} The work of CL was partially funded by the ANR project Smooth ANR-22-CE40-0017. CL would like to thank Martin Hairer and Lorenzo Zambotti for fruitful discussions on algebraic aspects of the theory of regularity structures. In addition, we would like to thank the two anonymous referees for their careful reading of our article.

 
\section{Self-adjoint generator of unbounded symmetric semigroup}\label{sec:semigroup}

This section presents a general result due to Klein and Landau. To ease the reading, let us anticipate its application to the context of (PAM): we will take $P_t f$ as the solution $u(t,\cdot)$ of the PDE \eqref{Eq:PAM} starting from $f$ and the domain $\cD_t$ will consist of a deterministic set of functions $f\in L^2(\R^d)$ which decay sufficiently fast at infinity; let us emphasise that the decay rate will depend on $t$ so that the sets $\cD_t$ will decrease with time. In this context, it seems quite difficult to extract the resolvents from the semigroup. Instead, Klein and Landau construct the spectral measures to establish the following theorem.

\begin{theorem}[\cite{KL81}]\label{Th:Klein-Landau}
	\BLUE{Fix $T>0$.} Let $(P_t)_{t \in [0, T]}$, be a collection of unbounded operators on a Hilbert space $\mathfrak{H}$ such that $P_0 = I$ and $P_t$ has domain $\cD_t$. Assume $(P_t)_{t \in [0, T]}$ satisfies the following properties:
	\begin{enumerate}
		\item (Non-increasing domains with dense union) For $s \leq t$, $\cD_{t} \subset \cD_{s}$ and $\bigcup_{0 < t \leq T} \cD_t$ is dense in $\mathfrak{H}$.
		\item (Semigroup) For $0 \leq s \leq t \leq T$, $P_s \cD_t \subset \cD_{t-s}$ and $P_{t-s} P_s = P_t$ on $\cD_t$.
		\item (Symmetry) For $f, g \in \cD_t$, $\crochet{f, P_t g}_\mathfrak{H} = \crochet{P_t f, g}_\mathfrak{H}$.
		\item (Weak continuity) For $t \in (0, T]$ and $f \in \cD_t$, the function $s \mapsto \crochet{f, P_s f}_\mathfrak{H}$ is continuous on the interval $[0, t]$.
	\end{enumerate}
	
	Then, there exists a unique self-adjoint operator $H$ on $\mathfrak{H}$ such that for every $t\in [0,T]$, $\cD_t \subset \cD(e^{-tH})$ and such that $P_t$ is the restriction of $e^{-tH}$ to $\cD_t$. Moreover, the operator $H$ is essentially self-adjoint on the core $\hat{\cD} := \bigcup_{0 < t \leq S} \bigcup_{0 < s < t} P_s \cD_t$, for any $S \in (0, T]$.
\end{theorem}

	Let us summarize the main steps of the proof. The key observation is the following: for any given $f \in \cD_t, t \in (0, T]$, the assumptions of the theorem ensure that the map $r:s \mapsto \norm{P_{s/2}f}^2_{\mathfrak{H}}$ is continuous, non-negative on $[0, t]$ and satisfies what the authors of \cite{KL81} call \emph{Osterwalder-Schrader positivity}, or \emph{reflection positivity} in the mathematical physics community: for all $(t_i)_{1\le i \le n} \in [0,T]$ and all $(c_i)_{1\le i \le n} \in\R$
	$$ \sum_{j,k=1}^n c_j c_k r(t_j+t_k) \ge 0\;.$$
	A theorem of Widder~\cite{Widder}, which can be thought of as a real version of the theorem of Bochner on characteristic functions of probability measures, implies that the map $r$ can be represented as the Laplace transform of a unique positive measure $\nu$ on $\R$, i.e. $\norm{P_{s/2} f}^2_{\mathfrak{H}} = \int_\R e^{-sa} \nu(\dd a)$. Let us point out that $\nu$ turns out to be the spectral measure (associated to $f$) of the self-adjoint operator $H$ to be defined.
	
	Since the Laplace transform of $\nu$ exists for $s \in [0, t]$, we can consider $\int_\R e^{-za} \nu(\dd a)$ for all $z \in \C$ with $\Re z \in [0, t]$. This allows to analytically extend the $\mathfrak{H}$-valued function $s \mapsto P_s f$ into a function $F_f$ defined on $\{z \in \C: \Re z \in [0, t]\}$ in such a way that $F_f(s) = P_s f$ for $s \in [0, t]$ and that $\crochet{F_f (z_1), F_f(z_2)}_\mathfrak{H} = \int_\R e^{-(\bar{z_1} + z_2)a} \nu(\dd a)$ for any $z_1, z_2 \in \C$ with $\Re(z_j) \in [0, T]$. As one can show that $F_f(z)$ is linear in $f$, it makes sense to set $U(y) f = F_f (iy)$ for all $f \in \cD_t, y \in \R$ and one has $\norm{U(y) f}_\mathfrak{H}^2 = \norm{F_f(0)}_\mathfrak{H}^2 = \norm{f}_\mathfrak{H}^2$. So far the construction works simultaneously for all $t \in (0, T]$, $U(y)f$ is therefore well-defined for all $f$ in the dense subspace $\bigcup_{0 < t \leq T} \cD_t$. We can thus continuously extend $U(y)$ into a unitary operator on $\mathfrak{H}$. One can then show that $(U(y))_{y \in \R}$ forms a strongly continuous one-parameter group of unitary operators. An application of Stone's Theorem then yields a unique self-adjoint operator $H$ such that $e^{-iy H} = U(y)$. By continuity $e^{-sH}f = P_s f$ for all $s\in [0,t]$.
	
	\BLUE{Let us finally observe that the operator $H$ defined in Theorem \ref{Th:Klein-Landau} does not depend on $T$. Indeed, if we consider two distinct parameters $T_1$ and $T_2$, the two operators $H_1$ and $H_2$ defined by the theorem are both essentially self-adjoint on the core $\hat{\cD}$, where we pick $S= T_1 \wedge T_2$, and this suffices to deduce that $H_1=H_2$.}
	

\section{The Anderson Hamiltonian in dimension two}\label{sec:2d}

The goal of this section is to construct the Anderson Hamiltonian with white noise potential on $\R^2$. To that end, we will feed Theorem \ref{Th:Klein-Landau} with the solution to the parabolic Anderson model
\begin{equation}
	\begin{cases}\partial_t u = \Delta u - u\xi\;,\quad \mbox{ on }\R^2\;,\\
	u(t=0,\cdot) = f(\cdot)\;.
	\end{cases}
\end{equation}
An analysis of the regularities at stake shows that this PDE is ill-defined, and, actually requires renormalization by infinite constants. Let us explain how it can be performed. Recall the smoothed out noise $\xi_\eps$ of the introduction. Denote by $u_\eps$ the solution of the above PDE in which $-u\xi$ is replaced by $-u_\eps(\xi_\eps - C_\eps)$ for some well-chosen constant $C_\eps$. Set $G\BLUE{(x)}= -\frac{\log|x|}{2\pi} \chi(x)$, \BLUE{$x\in\R^2$,} where $\chi$ is a smooth, radial, cutoff function equal to $1$ in the unit ball $B(0, 1)$ and supported in $B(0, 2)$. The function $G$ coincides with the Green's function of the Laplacian in the unit ball and one can check that $-\Delta G = \delta_0 + F$ for some smooth function $F$ supported in $B(0,2)\backslash B(0, 1)$.\\
We rely on the change of unknown $w_\eps := e^{Y_\eps} u_\eps$ where $Y_\eps := G*\xi_\eps$. It turns out that $w_\eps$ solves the following PDE
\begin{equation}\label{eq:PAM_weps}
	\begin{cases} \partial_t w_\eps = \Delta w_\eps - 2\nabla Y_\eps \cdot \nabla w_\eps + ( \vert \nabla Y_\eps\vert^2 - C_\eps + F*\xi_\eps)  w_\eps \;,\quad t>0\;, x\in \R^2\;,\\
		w_\eps(t=0,\cdot)= e^{Y_\eps} f\BLUE{(\cdot)} \;.
	\end{cases}
\end{equation}
The main observation is that there exists an appropriate choice of $C_\eps$ for which $ \vert \nabla Y_\eps\vert^2 - C_\eps$ converges in probability as $\eps\downarrow 0$ to some well-defined random distribution. In addition, the regularities at stake are such that the limit as $\eps\downarrow 0$ of this PDE is well-posed: the change of unknown ``improved'' the regularity of the random terms on the r.h.s. The enhanced noise in this specific case must include the limit in probability of $\vert \nabla Y_\eps\vert^2 - C_\eps$.\\

In Subsection \ref{subsec:Enhanced2d} we introduce the space of enhanced noises and perform the renormalisation: it corresponds to Step 1 of the introduction. In Subsection \ref{subsec:PAM} we construct the solution theory to a PDE corresponding to \eqref{eq:PAM_weps}: this implements Step 2 of the introduction. In Subsection \ref{subsec:semigroup2d} we define the semigroup and its generator: this is Step 3 of the introduction. Finally in Subsection \ref{subsec:supports2d} we characterise the support of the enhanced noise. In this whole section, $\kappa > 0$ is a small parameter (it is implicitly taken as small as needed).

\subsection{The space of enhanced noise and the white noise case}\label{subsec:Enhanced2d}

We consider the space of enhanced noises $\ccM := \cC^{-1-\kappa}_{p_\kappa} \times \cC^{-\kappa}_{p_\kappa}$ and denote by $q := (X,U)$ a generic element in this space. Endowed with the norm $\norm{q} = \norm{X}_{\cC^{-1-\kappa}_{p_\kappa} }+ \norm{U}_{\cC^{-\kappa}_{p_\kappa}}$, $\ccM$ is a separable Banach space.\\

\BLUE{Let us motivate this choice of space of enhanced noises on} a specific example. Fix some parameter $b \in (0,\kappa/2)$ and set $\Omega := \cC^{-1-\kappa}_{p_b}(\R^2)$ endowed with the law $\P$ of white noise. The canonical variable on $\Omega$ will be denoted $\xi$. We define $\xi_\eps := \xi * \varrho_\eps$, $Y_\eps := G* \xi_\eps$ and
\[ Z_\eps(x):= \vert \nabla Y_\eps(x) \vert^2 - C_\eps\;,\quad  C_\eps := \E\big[\vert \nabla Y_\eps(0) \vert^2\big]\;.\]
Note that $x\mapsto Y_\eps(x)$ is a stationary (\BLUE{G}aussian) field so that $C_\eps$ is left unchanged if we shift the point at which $\nabla Y_\eps$ is evaluated. We set $Q_\eps := (\xi_\eps,Z_\eps)$, this r.v.~takes values in $\ccM$, see for instance~\cite[Lemma 1.1 and Corollary 1.2]{HL15}. Let us give some details regarding $\xi_\eps$ since it illustrates the role played by the polynomial weight here (and since this argument will appear at several occasions later on): using \eqref{Eq:Calphaw} we find for any $n\ge 1$
\begin{equation}\label{Eq:xieps2d}
	\E\left[ \| \xi_\eps\|_{\cC^{-1-\kappa}_{p_\kappa}}^n \right] \le C \sum_{k\in\Z^d} \frac{\E\left[ \| \xi_\eps  \chi(\cdot-k)\|_{\cC^{\BLUE{-1-\kappa}}}^n \right]}{p_\kappa(k)^n}\;.
\end{equation}
The term $\E\left[ \| \xi_\eps  \chi(\cdot-k)\|_{\cC^{\BLUE{-1-\kappa}}}^n \right]$ is independent of $k$ and is finite for any $n\ge 1$: therefore, choosing $n$ large enough, the previous sum converges.\\
Note that $Z_\eps$ ``contains'' two instances of $\xi$ so that it requires a weight of order $p_{2b}$, and since we chose $b < \kappa/2$, it indeed lives in a space with weight $p_\kappa$.\\

For any $x\in \R^2$, we define the shift operator $\theta_x$ on $\Omega$ as follows: $\theta_x \xi(\cdot) := \xi(\cdot-x)$, or more formally
$$ \langle \theta_x \xi , \varphi\rangle = \langle \xi , \varphi(\cdot+x) \rangle\;,\quad \varphi \in C^\infty_c(\R^2)\;.$$
It is easy to check that the shift operators are continuous maps from $\Omega$ into itself. We naturally extend this definition to $\ccM$ by defining $\theta_x q := (\theta_x X, \theta_x U)$ for any $q=(X,U) \in \ccM$. The shift operators are continuous maps from $\ccM$ into itself.\\

To identify the spectrum of $\cH$ in Section \ref{sec:spectrum}, we will need to show a commutation property: the operator $\cH(\theta_x \xi)$ associated to the shifted noise coincides with the operator $\cH(\xi)$ conjugated with spatial shifts, see Lemma \ref{lem:shiftshift}. Therefore we need to argue that one can deal simultaneously with all the shifted noises: this is the purpose of the following technical lemma.

\begin{lemma}\label{lem:Omega02d}
	There exists a sequence $\eps_k \downarrow 0$ such that the set
	\begin{equation}\label{Eq:Omega0_2d}
		\Omega_0 : =\Big\{\xi\in\Omega: Q_{\eps_k}(\theta_x \xi) \mbox{ converges in $\ccM$ as } k\to\infty\;,\quad \forall x\in \R^2\Big\}\;,
	\end{equation}
	is of full $\P$-measure and is invariant under all $\theta_x, x\in \R^2$. The\footnote{$Q$ is defined as a limit on the set of full measure $\Omega_0$, and is extended to $\Omega$ arbitrarily.} limit $Q$ satisfies $Q(\theta_x \xi) = \theta_x Q(\xi)$ for all $\xi \in\Omega_0$.\\
	As $\eps\downarrow 0$, the field $(Q_\eps(\theta_x \xi))_{x\in \R^2}$ which takes values in $\ccM$ converges in probability, locally uniformly over $x\in \R^2$, to $(Q(\theta_x \xi))_{x\in \R^2}$.
\end{lemma}
\begin{proof}
	We start with a general bound. Let $f\in \cC^{\beta}_{p_\kappa}$ for some $\beta < 0$. Then for any $x\in \R^d$
	$$ \| \theta_x f \|_{\cC^{\beta}_{p_\kappa}} \le \BLUE{\Big( \sup_{y\in\R^d} \frac{p_\kappa(y-x)}{p_\kappa(y)}  \Big)\;} \| f \|_{\cC^{\beta}_{p_\kappa}} = p_\kappa(x) \| f \|_{\cC^{\beta}_{p_\kappa}}\;.$$
	This inequality is a direct consequence of the definitions of the shift and of the $\cC^{\beta}_{p_\kappa}$-norm. It implies that
	\begin{equation}\label{Eq:BoundShiftedf}
		\sup_{x\in \R^2} \frac1{p_\kappa(x)}\| \theta_x f \|_{\cC^{\beta}_{p_\kappa}} \le \| f \|_{\cC^{\beta}_{p_\kappa}}\;.
	\end{equation}
	It is straightforward to check that for any $\eps\in (0,1)$, any $x\in\R^2$ and any $\xi\in\Omega$
	$$ Q_\eps(\theta_x \xi) = \theta_x Q_\eps(\xi)\;.$$
	This observation combined with\BLUE{~\eqref{Eq:BoundShiftedf}} implies that for all $\xi\in \Omega$ and all $\eps,\eps' \in (0,1)$
	$$\sup_{x\in \R^2} \frac1{p_\kappa(x)} \| Q_\eps(\theta_x \xi) - Q_{\eps'}(\theta_x \xi) \|_{\ccM} \le \| Q_\eps(\xi) - Q_{\eps'}(\xi) \|_{\ccM}\;.$$
	The arguments in \cite[Lemma 1.1, Proposition 1.3]{HL15} show that, for any given $p>1$, $Q_\eps(\xi)$ converges in $L^p(\Omega,\P)$ as $\eps\downarrow 0$ to some limit $Q(\xi) = (\xi,Z)$ in $\ccM$. We thus deduce that the collection of r.v.~$(Q_\eps(\theta_x \xi),x\in\R^2)$, taking values in $L^\infty_{p_\kappa}(\R^2 \to \ccM)$, converges as $\eps\downarrow 0$ in $L^p(\Omega,\P)$. Therefore there exists a deterministic sequence $\eps_k \downarrow 0$ such that the convergence holds almost surely. We thus consider the set of full $\P$-measure
	$$ \Omega_0 := \Big\{\xi\in\Omega: Q_{\eps_k}(\theta_x \xi) \mbox{ converges in $\ccM$ as }k\to\infty\;,\quad \forall x\in \R^2\Big\}\;.$$
	It is elementary to check that $\Omega_0$ is invariant under all shifts $\theta_x$, $x\in\R^2$. Let us denote by $\tilde{Q}(x,\xi) := \lim_k Q_{\eps_k}(\theta_x \xi)$ for all $x\in\R^2$ and all $\xi\in\Omega_0$. For any $\xi\in\Omega_0$ and any $x\in \R^2$
	$$ \tilde{Q}(x,\xi) = \lim_k Q_{\eps_k}(\theta_x \xi) = \lim_k Q_{\eps_k}(\theta_0(\theta_x \xi)) = \tilde{Q}(0,\theta_x \xi)\;,$$
	and, using the continuity of the shift operator on $\ccM$
	$$ \tilde{Q}(x,\xi) = \lim_k Q_{\eps_k}(\theta_x \xi) = \lim_k \theta_x Q_{\eps_k}(\xi) = \theta_x \tilde{Q}(0,\xi)\;.$$
	If we now set $Q(\xi) := \tilde{Q}(0,\xi)$, we have shown that for all $\xi\in\Omega_0$
	$$ Q(\theta_x \xi) = \theta_x Q(\xi)\;.$$
	To conclude the proof, let us assign arbitrary values to $Q(\xi)$ for $\xi\in \Omega\backslash\Omega_0$. We recall that $(Q_\eps(\theta_x \xi),x\in\R^2)$ converges in $L^p(\Omega,\P)$. Necessarily its limit coincides with the a.s.~limit along $\eps_k$, that is $(Q(\theta_x\xi), x\in\R^2)$.
\end{proof}

\subsection{The parabolic evolution}\label{subsec:PAM}

For a given $q=(X, U) \in \ccM$, we set $V := G*X$. One can check that $V$ belongs to $\cC^{1-\kappa}_{p_\kappa}$. Given the \emph{enhanced noise} $q=(X,U)$, we consider the PDE
\begin{equation}\label{eq:PAM_w}
	\begin{cases} \partial_t w = \Delta w - 2\nabla V \cdot \nabla w + (U + F*X)  w \;,\quad t>0\;, x\in \R^2\;,\\
		w(t=0,\cdot) = w_0\;.
	\end{cases}
\end{equation}

\begin{remark}
	If one takes $X=\xi_\eps$ and $U = \vert \nabla V\vert^2 - C_\eps$, then $w$ coincides with $w_\eps$ of \eqref{eq:PAM_weps}.
\end{remark}

\begin{theorem}\label{Th:PAM2d}
	The mild solution to \eqref{eq:PAM_w} yields a map $(q,f,t) \mapsto w^{q, f}(t, \cdot)$ defined on $\ccM\times  \Big( \bigcup_{\ell_0\in\R} L^2_{e_{\ell_0}}\Big) \times \R_+$ such that for any $\ell_0\in\R$:
	\begin{enumerate}
		\item (Continuity in the data) For any $t \ge 0$, the map $(q, f) \mapsto w^{q, f}(t, \cdot)$ is continuous from $\ccM \times L^2_{e_{\ell_0}}$ into $L^2_{e_{\ell_0 + t}}$, and is linear w.r.t.~$f$.
		\item (Continuity in time) For any $(q,f) \in \ccM\times L^2_{e_{\ell_0}}$ and any $T>0$, the map $t \mapsto w^{q, f}(t, \cdot)$ is continuous from $[0,T]$ into $L^2_{e_{\ell_0 + T}}$.
		\item (Semigroup property) For any $(q,f) \in \ccM\times L^2_{e_{\ell_0}}$ and any $0\le s \le t$, $w^{q, f}(t, \cdot)$ coincides with $w^{q,g}(s,\cdot)$ where $g = w^{q,f}(t-s,\cdot)$.
	\end{enumerate}
\end{theorem}

\begin{proof}
	This is essentially an adaptation of \cite{HL15}, except for property 2. Let us present the main arguments. Given a time horizon $T>0$, and some parameters $\ell \in \R$, $\alpha, \beta \ge 0$ we define $\cE_{\ell, T, \alpha, \beta}$ as the completion of all smooth functions $u: (0, T] \times \R^2 \to \R$ under the norm
	\begin{equation}
	\normm{u}_{\ell, T, \alpha, \beta} := \sup_{t \in (0, T]} t^{\beta} \norm{u(t, \cdot)}_{\cH^{\alpha}_{e_{\ell + t}}} < \infty \;.
	\end{equation}
	Set $g = U+F*X$. Since $F$ is smooth, we have $g \in \cC^{-\kappa}_{p_\kappa}$. For simplicity we write $w(t) = w(t, \cdot)$ and we set
	\begin{equation}\label{eq:fixed-point}
		\cM_{T, f} w(t) := K_t*f + \int_0^t K_{t-s}*(-2\nabla V \cdot \nabla w(s) + gw(s)) \dd s,  \quad t \in (0, T]
	\end{equation}
	with $K_t$ denoting the heat kernel. This corresponds to the mild formulation of \eqref{eq:PAM_w}. We are going to show that $\cM_{T, f}$ admits a unique fixed point in $\cE_{\ell_0, T, \alpha, \beta}$ provided $T$ is small enough.\\
	Set $\alpha = 1 + 2\kappa$ and $\beta = \frac{\alpha}{2} = \frac12 + \kappa$. \BLUE{By standard product results on functions and distributions spaces, we have uniformly over all $s \in (0,T]$
		$$ \| \nabla V \cdot \nabla w(s) \|_{\cH^{\alpha-1}_{p_\kappa e_{\ell_0+s}}} \lesssim s^{-\beta} \norm{q}_\ccM \normm{w}_{\ell_0, T, \alpha, \beta}\;,$$
	and
	$$ \| gw(s) \|_{\cH^{-\kappa}_{p_\kappa e_{\ell_0+s}}} \lesssim s^{-\beta} \norm{q}_\ccM \normm{w}_{\ell_0, T, \alpha, \beta}\;.$$
	Consequently by} the Schauder estimate for the heat kernel and the bound \eqref{Eq:TrickWeights}, one has
	\begin{align*}
	\norm{K_t*f}_{\cH^{\alpha}_{e_{\ell_0}}} &\lesssim t^{-\beta}\norm{f}_{L^2_{e_{\ell_0}}}\;,\\
	\norm{K_{t-s}*(-2\nabla V \cdot \nabla w(s) + gw(s))}_{\cH^{\alpha}_{e_{\ell_0+t}}} &\lesssim (t-s)^{-\frac{\alpha+\kappa}{2}-\kappa} s^{-\beta} \norm{q}_\ccM \normm{w}_{\ell_0, T, \alpha, \beta}\;.
	\end{align*}
	Since $\kappa$ can be chosen arbitrarily small, one can make $ \frac{\alpha}{2} + \frac{3\kappa}{2} < 1$ and $\beta <1$. The integral over time in \eqref{eq:fixed-point} therefore converges and is of order $\int_0^t (t-s)^{-\frac{\alpha+\kappa}{2}-\kappa} s^{-\beta} \dd s  \lesssim t^{1 - \frac{\alpha}{2} - \frac32 \kappa} t^{-\beta}$,
	whence
	\begin{equation} \label{eq:bound_w(t)}
	\normm{\cM_{T, f} w}_{\ell_0, T, \alpha, \beta} \lesssim  \norm{f}_{L^2_{e_{\ell_0}}} + T^{1 - \frac{\alpha}{2} - \frac32\kappa} \norm{q}_\ccM \normm{w}_{\ell_0, T, \alpha, \beta}.
	\end{equation}
	Therefore, $\cM_{T, f}$ is a well-defined map on $\cE_{\ell_0, T, \alpha, \beta}$. In the same vein, for any two elements $w, \bar{w} \in \cE_{\ell_0, T, \alpha, \beta}$, the previous calculation shows
	\[\normm{\cM_{T, f}(w - \bar{w})}_{\ell_0, T, \alpha, \beta} \lesssim T^{1 - \frac{\alpha}{2} - \frac32\kappa} \norm{q}_\ccM \normm{w - \bar{w}}_{\ell_0, T, \alpha, \beta}.\]
	Since $1 - \frac{\alpha}{2} - \frac32\kappa > 0$, $\cM_{T, w_0}$ can be made contracting by choosing a $T$ small enough (depending only on $q$), in which case there exists a unique fixed point in $\cE_{\ell_0, T, \alpha, \beta}$ that we denote $w^{q,f}$.  As $T$ depends only on $q$, this solution map can be extended to global in time solution by an iterative argument consisting in restarting the equation from the initial data $w^{q, f}(T)$, $w^{q,f}(2T)$, etc. We have now proved the existence and uniqueness of the solution map $(q,f,t) \mapsto w^{q,f}(t)$.\\

	The semigroup property stated as point 3. follows from the semigroup property of the heat kernel and from the equation satisfied by the unique fixed point of the map defined above. The linearity in $f$ is immediate from \eqref{eq:fixed-point}. Now we prove the continuity of $(q, f) \mapsto w^{q, f}$. Fix any $R>0$ and take two elements $(q_j = (X_j, U_j), f_j)$, $j = 1, 2$ in the ball of radius $R$ of $\ccM \times L^2_{e_{\ell_0}}$. Denote $V_j = G*X_j$ and $g_j = U_j + F*X_j$, one has
	\begin{align*}
	w^{q_1, f_1}(t) &- w^{q_2, f_2}(t) = K_t *(f_1 - f_2) + \int_0^t K_{t-s}*[-2\nabla (V_1 - V_2) \cdot \nabla w^{q_1, f_1}(s)\\
	&- 2\nabla V_2 \cdot \nabla (w^{q_1, f_1} - w^{q_2, f_2}) + (g_1 - g_2)w^{q_1, f_1}(s) + g_2(w^{q_1, f_1} - w^{q_2, f_2})(s)] \dd s.
	\end{align*}
	A calculation similar to the one above shows
	\begin{align*}
	&\normm{w^{q_1, f_1} - w^{q_2, f_2}}_{\ell_0, T, \alpha, \beta} \lesssim \norm{f_1 - f_2}_{L^2_{e_{\ell_0}}} \\
	&+ T^{1 - \frac{\alpha}{2} - \frac32\kappa} \left(\norm{q_1 - q_2}_\ccM \normm{w^{q_1, f_1}}_{\ell_0, T, \alpha, \beta} + R \normm{w^{q_1, f_1} - w^{q_2, f_2}}_{\ell_0, T, \alpha, \beta} \right).
	\end{align*}
	By choosing $T$ small enough (depending only on $R$), we then deduce $(q, f) \mapsto w^{q, f}$ is uniformly continuous on the ball with radius $R$ in $\ccM \times L^2_{e_{\ell_0}}$. Since $R > 0$ is arbitrary, we have thus proved the continuity for fixed $t$ small enough. Again by an iterative argument, this can be extended to any $t > 0$.
	
	We turn to point 2., the strong continuity in time of $w=w^{q,f}$. First we observe that the heat operator $K_t$ is a bounded operator on $L^2_{e_\ell}$ for any given $\ell\in\R$, and that it is strongly continuous in $t$.\\	
	We now prove the right continuity. Fix $t \in [0, T)$. For $\eps>0$ such that $t + \eps\leq T$, we compute\BLUE{
	\begin{align*}
		w(t+\eps) &:= K_\eps\Big(K_t*f + \int_0^t K_{t-s}*(-2\nabla V \cdot \nabla w(s) + gw(s)) \dd s\Big)\\
		&\quad+ \int_t^{t+\eps}  K_{t+\eps-s}*(-2\nabla V \cdot \nabla w(s) + gw(s)) \dd s\\
		&:= K_\eps w(t) + \int_0^{\epsilon} K_{\eps-s}*(-2\nabla V \cdot \nabla w(t+s) + gw(t+s)) \dd s
	\end{align*}
	and therefore
}

	\begin{align*}
	&\norm{w(t+\epsilon) - w(t)}_{L^2_{e_{\ell_0+t+\eps}}} \leq \norm{K_\eps w(t) - w(t)}_{L^2_{e_{\ell_0+t+\eps}}}\\
	&+ \norm{\int_0^{\epsilon} K_{\eps-s}*(-2\nabla V \cdot \nabla w(t+s) + gw(t+s)) \dd s}_{L^2_{e_{\ell_0+t+\eps}}}
	\end{align*}
	The second term is controlled by \[\int_0^\eps (\eps - s)^{-\frac{3}{2}\kappa} (t+s)^{-\beta} \norm{q}_\ccM \normm{w}_{\ell_0, T, \alpha, \beta} \dd s \lesssim \eps^{1- \frac32\kappa}.\]
	The right continuity thus follows from the fact that $1 -\frac32\kappa > 0$ and the strong continuity of the heat kernel.\\
	For the left continuity, fix $t \in (0, T]$ and write $w(t) - w(t-\eps) = I_1 + I_2 + I_3$ where
	\begin{align*}
	&I_1 := (K_t - K_{t-\eps}) f\;,\quad I_2 := \int_0^{t-\eps} (K_{t-s} - K_{t-\eps-s})(-2\nabla V \cdot \nabla w(s) + gw(s)) \dd s\;, \\
	&I_3:= \int_{t-\eps}^t K_{t-s} (-2\nabla V \cdot \nabla w(s) + gw(s)) \dd s\;.
	\end{align*}
	Obviously the $L^2_{e_{\ell_0 + t}}$-norm of $I_1$ goes to $0$ as $\eps \to 0$ by the strong continuity of \BLUE{$K_\eps$} on $L^2_{e_{\ell_0+t}}$. For $I_3$, we have a similar control as previously: \[\norm{I_3}_{L^2_{e_{\ell_0+t}}} \lesssim \int_{t-\eps}^t (t-s)^{-\frac{3}{2} \kappa} s^{-\beta} \norm{q} \normm{w}_{\ell_0, T, \alpha,\beta} \dd s  \lesssim \eps^{1-\frac{3}{2} \kappa}(t-\eps)^{-\beta}.\]
	For $I_2$, we use $K_{t-s} - K_{t-s-\eps} = K_{(t-s)/2-\eps} (K_\eps - I) K_{(t-s)/2}$ and that $\norm{K_{(t-s)/2-\eps}} \leq 1$ as a bounded operator on $L^2_{e_\ell}$ for any given $\ell \in \R$. It gives 
	\begin{align*}
	\norm{I_2}_{L^2_{e_{\ell_0+t}}} \leq \int_0^{t-\eps} h_\eps(s) \dd s
	\end{align*}
	where \[h_\eps(s) = \norm{(K_\eps - I)K_{\frac{t-s}{2}}(-2\nabla V\cdot w(s) + gw(s))}_{L^2_{e_{\ell_0 + t}}}, \quad s\in [0, t].\]
	Note that by the strong continuity of $K_\eps$, one has $h_\eps(s) \to 0$ for all $s \in [0, t]$; moreover, uniformly \BLUE{over} $\eps$ we have $h_\eps(s) \lesssim (t-s)^{-\frac{3}{2}\kappa}s^{-\beta} \norm{q}_{\ccM} \normm{w}_{\ell_0, T, \alpha, \beta}$ which is an $L^1$-function on $[0, t]$. By dominated convergence it follows that $\norm{I_2}_{L^2_{e_{\ell_0+t}}} \leq \int_0^{\BLUE{t-\eps}} h_\eps(s) \dd s \to 0$ as $\eps \to 0$. Since $I_1, I_2$ and $I_3$ all converge to $0$ as $\eps \to 0$, we have thus proved the left-continuity of $w$.
\end{proof}

\subsection{The semigroup}\label{subsec:semigroup2d}

\BLUE{For $q\in\ccM$, recall that $V = G*X$ so that there exists $C>0$ such that $|V(x)| \leq C (1+|x|)^{\kappa}$ for all $x\in\R^2$. Consequently, for any given $\delta >0$, it holds $e^{|V|} \lesssim e^{\delta(1+|x|)}$ and thus $e^{|V|} \in L^{\infty}_{e_{\delta}}$. Based on this observation, and on Theorem \ref{Th:PAM2d} the definition below makes sense.}

\begin{definition}\label{def:symmetric-semigroup}
	Fix $\delta > 0$ and $q \in \ccM$. For any $t > 0$, define the domain $\cD_t = \cD_t(q) := e^{-V} L^2_{e_{-t - \delta}} \subset L^2$ and the operator
		\[P_t : \bigg\{ \begin{array}{lll}
		\cD_t &\to &L^2\\
		f & \mapsto &e^{-V} w^{q,e^{V} f}(t)
	\end{array}\]
	Define also $P_0 = I$ on $L^2$.
\end{definition}

\begin{remark}
	If one sets $u=e^{-V} w^{q,e^V f}$ then $u$ formally solves
	\begin{equation}\label{eq:PAM2d}
		\begin{cases}
			\partial_t u = \Delta u + (U - X - |\nabla V|^2)  u, \quad t>0, ~ x\in \R^2,\\
			u(t=0,\cdot) = f \;.
		\end{cases}
	\end{equation}
	This PDE makes sense when $X, U$ are functions, but is only formal in general as the terms $|\nabla V|^2$, $Xu$ are singular products.
\end{remark}

In the sequel, we regard $P_t$ as an unbounded operator on the Hilbert space $\mathfrak{H} = L^2(\R^2,dx)$ with domain $\cD_t$. The following proposition shows that for any given $T>0$, the collection $(P_t)_{t \in [0, T]}$ satisfies the assumptions of Theorem \ref{Th:Klein-Landau}.
\begin{proposition}
	Fix $T>0$. The collection $(P_t)_{t \in [0, T]}$ is a semigroup of symmetric operators on $\mathfrak{H} = L^2(\R^2,dx)$ and is strongly continuous with respect to $t$. In particular, it satisfies the \emph{(Non-increasing domains with dense union)}, \emph{(Semigroup)}, \emph{(Symmetry)} and \emph{(Weak continuity)} properties of Theorem \ref{Th:Klein-Landau}.
\end{proposition}
\begin{proof}
	{\emph{Non-increasing domains with dense union}:} For $0 \leq s \leq t \leq T$, we have obviously $\cD_t = e^{-V} L^2_{e_{-t - \delta}} \subset e^{V} L^2_{e_{-s - \delta}} = \cD_s$. Moreover, the set $e^{-V} C^\infty_c$ is contained in $L^2_{e_{-t - \delta}}$, therefore $C^\infty_c \subset \cD_t$ for all $t \in (0, T]$ so that $\cD_t$ is dense in $L^2$.\\
	{\emph{Semigroup}:} Let $f \in \cD_t$. The first property of Theorem \ref{Th:PAM2d} shows that $g:= w^{q,e^V f}(s) \in L^2_{e_{-(t-s) - \delta}}$ and therefore $P_s f = e^{-V} g \in \cD_{t-s}$. In addition the third property of Theorem \ref{Th:PAM2d} shows that
	$$P_t f = e^{-V} w^{q,e^{V} f}(t) = e^{-V} w^{q,g}(t-s)\;,$$
	that is $P_t f = P_{t-s} (e^{-V} g) = P_{t-s}(P_{s} f)$ as required.\\
	{\emph{Weak continuity}:} This is a direct consequence of the second property of Theorem \ref{Th:PAM2d}. Actually it even proves strong continuity.\\
	{\emph{Symmetry}:} By the continuity properties stated in Theorem \ref{Th:PAM2d}, it suffices to consider $q=(X,U)$ where $X,U$ are smooth functions with compact support. The symmetry property can be restated at the level of $w$ in the following way
	$$ \int_{x\in\R^2} w^{q,e^V f}(t,x) e^{-V(x)} g(x) dx = \int_{x\in\R^2} w^{q,e^V g}(t,x) e^{-V(x)} f(x) dx\;.$$
	These two terms are the values at $s=0$ and $s=t$ of the map
	$$h(s) := \int_{x\in\R^2} w^{q,e^V f}(s,x) e^{-2V(x)} w^{q,e^V g}(t-s,x) dx\;.$$
	Since $X,U$ are smooth, $w^{q,e^V f}$ and $w^{q,e^V g}$ are strong solutions of \eqref{eq:PAM_w} and a direct computation shows that $h(s)$ is constant in $s$.
\end{proof}

\begin{definition}\label{def:L}
	For $q \in \ccM$, let $H(q)$ be the unique self-adjoint operator associated, thanks to Theorem \ref{Th:Klein-Landau}, to the symmetric semigroup $(P_t)_{t \in [0, T]}$ of Definition \ref{def:symmetric-semigroup}. In particular, $P_t$ is the restriction of $e^{-tH(q)}$ to $\cD_t$ for all $t \in [0, T]$.
\end{definition}

\begin{remark}
	Since $\cD_t$ is dense in $L^2$, the proof of~\cite[Lemma 6]{KL81} shows that $H(q)$ is essentially self-adjoint over $\hat{\cD}(q) := \bigcup_{0 < s < t} P_s \cD_t$. Furthermore, from the proof of Theorem \ref{Th:PAM2d}, one can deduce that $\hat{\cD}(q) \subset \cH^{1-\kappa}(\R^2,dx)$.
\end{remark}

We now establish a few properties satisfied by $H(q)$. Given any $f\in L^2$, denote by $\mu_f(q)$ the spectral measure associated to the self-adjoint operator $H(q)$ and $f$, that is, the unique finite measure on $\R$ with Stieltjes transform
$$ \int_\R (\lambda-z)^{-1} \mu_f(q)(d\lambda) = \langle f, (H(q)-z)^{-1} f\rangle\;,\quad \forall z\in \C\backslash \R\;.$$

\begin{proposition}\label{prop:continuity_H}
	\begin{enumerate}
			\item For any $f\in L^2$, the map $q \mapsto \mu_f(q)$ is continuous from $\ccM$ into the space of finite non-negative measures endowed with the topology of weak convergence. As a consequence $q\mapsto H(q)$ is continuous in the strong resolvent sense.
			\item Let $q = (X, U) \in \ccM$ be such that $X,U \in L^\infty_{p_{b}}$. Then $H(q)$ coincides with the operator $-\Delta + X + |\nabla  G * X|^2 - U$ which is essentially self-adjoint over $C^\infty_c$.
		\end{enumerate}
\end{proposition}
\begin{remark}\label{Rk:Heps}
	The point 2 covers the case when $X = \xi_\eps$ and $U = |\nabla G * \xi_\eps|^2 - C_\eps$.
\end{remark}
\begin{proof}
	We start with the first item. Fix $f\in C_c$. Then $f\in \cD_t(q)$ for any $q\in \ccM$ and any $t\ge 0$. In addition, we have
	\[\int_\R e^{-s \lambda} \mu_f(q)(\dd \lambda) = \crochet{f, e^{-sH(q)} f} = \crochet{f, e^{-V} w^{q, e^V f}(s)}, \quad \forall s \in [0, t] \;.\]
	Suppose $(q_n)$ is a sequence in $\ccM$ converging to $q$. Since $f$ has compact support, it is easy to check that $e^{V(q_n)} f$ converges to $e^{V(q)}f$ in $L^2_{e_{\ell_0}}$ for any $\ell_0\in\R$. Then the first property of Theorem \ref{Th:PAM2d} ensures that $w^{q_n, e^{V(q_n)} f}(s)$ converges to $w^{q, e^{V(q)} f}(s)$ in $L^2_{e_{\ell_0+t}}$ for any $\ell_0\in\R$. Using again the fact that $f$ has compact support, we deduce that
	$$ \crochet{f, e^{-V(q_n)} w^{q_n, e^{V(q_n)} f}(s)} \to \crochet{f, e^{-V(q)} w^{q, e^{V(q)} f}(s)}\;,\quad n\to\infty\;.$$
	Therefore the Laplace transform of the measure $\mu_f(q_n)$ converges to the Laplace transform of $\mu_f(q)$. By Lemma \ref{Lemma:Laplace} below this implies weak convergence of the measures. In turn, it implies that their Stieltjes transforms converge: for any $z\in\C\backslash\R$
	$$ \int_\R (\lambda-z)^{-1} \mu_f(q_n)(\dd \lambda) \to \int_\R (\lambda-z)^{-1} \mu_f(q)(\dd \lambda)\;,\quad n\to\infty\;,$$
	that is $\crochet{f, (H(q_n)-z)^{-1} f} \to \crochet{f, (H(q)-z)^{-1} f}$. Since the resolvents at $z$ are uniformly bounded operators and since $C_c$ is dense in $L^2$, this last convergence remains true for all $f\in L^2$. In other words, we have shown weak resolvent convergence of $H(q_n)$ to $H(q)$. By~\cite[Lemma 6.37]{Tes14}, this implies strong resolvent convergence, as required.\\
	To conclude the proof of the first item, we show that $\mu_f(q_n)$ converges weakly to $\mu_f(q)$ for any given $f\in L^2$ (so far, we only proved it for $f\in C_c$). The weak resolvent convergence of $H(q_n)$ to $H(q)$ ensures that for any $f\in L^2$, the Stieltjes transform of $\mu_f(q_n)$ converges to the Stieltjes transform of $\mu_f(q)$. Since these measures have a finite mass (equal to the $L^2$ norm of $f$), this implies weak convergence.
	
	For the second item, we write $W = X + |\nabla G * X|^2 -U$ which is a locally bounded potential satisfying the bound $W(x) \ge - C p_\kappa(x)$ for some constant $C>0$. Since $\kappa<2$, the operator $T = -\Delta + W$ is essentially self-adjoint over $C^\infty_c$ \cite{Kat72}. Let us still denote by $T$ the self-adjoint extension. \BLUE{For any $L>0$, we set $W_L := W \mathbf{1}_{[-L/2,L/2]^d}$, and we let $T_L$ be the self-adjoint extension of $-\Delta + W_L$. An argument of Weidmann, see~\cite[Thm 1]{Wei97}, shows that $T_L$ converges to $T$ as $L\to\infty$ in the strong resolvent sense. We also set $q_L := (X_L,U_L)$ with $X_L := X \mathbf{1}_{[-L/2,L/2]^d}$ and $U_L := U \mathbf{1}_{[-L/2,L/2]^d}$. It is straightforward to check that $q_L$ converges to $q$ in $\ccM$ as $L\to\infty$. The first item proven above thus implies that $H(q_L)$ converges to $H(q)$ in the strong resolvent sense. To prove that $T=H(q)$, it suffices to prove that $T_L = H(q_L)$ for any $L>0$. To that end, we observe that $T_L$ is bounded from below, consequently $e^{-s T_L}$ is bounded from above for any $s \ge 0$. Fix $f \in L^2$ and set $u(s) = e^{-sT_L} f$ for $s \ge 0$, so that $u$ can be regarded as a function from $\R_+$ to $L^2$. It can be checked that $u$ solves \eqref{eq:PAM2d} with $q_L = (X_L, U_L)$. Then by the uniqueness stated in Theorem \ref{Th:PAM2d} in a weighted $L^2$ space, $u$ coincides with $u^{q_L,f} := e^{-G*X_L} w^{q_L,f}$. In turn, this implies that $u^{q_L,f}(s)$ lies in $L^2$ at any time, and therefore the operator $e^{-s H(q_L)}$ is necessarily a bounded operator on $L^2$. We have thus proven that $e^{-s T_L} = e^{-s H(q_L)}$, which suffices to deduce that $T_L = H(q_L)$.
	}
\end{proof}

\begin{lemma}\label{Lemma:Laplace}
	Let $\nu_n, \nu$ be finite measures on $\R$. Assume that there exists $T>0$ such that for all $t\in [0,T]$
	$$ \int_\R e^{-\lambda t} \nu_n(d\lambda) \to \int_\R e^{-\lambda t} \nu(d\lambda)\;,\quad n\to\infty\;,$$
	(implicitly we assume that all these quantities are finite). Then $\nu_n$ converges weakly to $\nu$ as $n\to\infty$.
\end{lemma}
\begin{proof}
	For any $z\in O_T := \{z\in \bbC: -T< \Re(z) < 0\}$, we set $g_n(z) := \int_\R e^{\lambda z} \nu_n(d\lambda)$ and $g(z) := \int_\R e^{\lambda z} \nu(d\lambda)$. By assumption $(g_n)_{n\ge 1}$ is a sequence of holomorphic functions on $O_T$ which is uniformly bounded for the supremum norm. By Montel's Theorem, we deduce that there exists a subsequence $g_{n_k}$ that converges uniformly on compact sets of $O_T$ to some holomorphic function $h$. By uniqueness of analytic continuation, $h$ must coincide with $g$ on $O_T$ and this suffices to deduce that the whole sequence $g_n$ converges to $g$ uniformly on compact sets of $O_T$. If we set $\mu_n(d\lambda) := e^{-\lambda T/2} \nu_n(d\lambda)$ and $\mu(d\lambda) := e^{-\lambda T/2} \nu_n(d\lambda)$, the previous argument impl\BLUE{ies that the characteristic function of $\mu_n$ converges to the characteristic function of $\mu$, so} that $\mu_n$ converges weakly to $\mu$. In turn, this implies that $\nu_n$ converges vaguely to $\nu$. Since the total mass of $\nu_n$ converges to the total mass of $\nu$, this suffices to conclude.
\end{proof}

We have all the ingredients at hand to define the Anderson Hamiltonian with white noise potential on $\R^2$. For any $\xi\in\Omega$, recall that $Q_\eps(\xi), Q(\xi) \in \ccM$ and set
	\[\cH_\eps(\xi) := H(Q_\eps(\xi))\;,\quad \cH(\xi) = H(Q(\xi))\;,\]
where $H$ is the deterministic map introduced in Definition \ref{def:L}.
\begin{proof}[Proof of Theorem \ref{Th:Construction} in dimension $2$]
	By Remark \ref{Rk:Heps}, for every $\xi\in\Omega$ we have $\cH_\eps(\xi) = -\Delta+\xi_\eps + C_\eps$. By Lemma \ref{lem:Omega02d}, $Q$ is the limit in probability of $Q_\eps$ so that Proposition \ref{prop:continuity_H} entails that $\cH_\eps$ converges in the strong resolvent sense to $\cH$ in probability. Finally, Definition \ref{def:L} ensures that, for any $t\ge 0$, the domain of the operator $e^{-t\cH}$ contains the set $\cD_t$ so in particular all functions $f\in L^2(\R^d,dx)$ with compact support, and $e^{-t\cH} f$ coincides with the solution of \eqref{Eq:PAM} at time $t$, starting from $f$ at time $0$.
\end{proof}

\subsection{Characterization of the support}\label{subsec:supports2d}

We conclude this section by proving a crucial result in order to identify the spectrum of $\cH$. Let $\supp(Q_\eps)$, resp.~$\supp(Q)$, be the topological support of the law of the r.v.~$Q_\eps$, resp.~$Q$, that is, the intersection of all closed sets of $\ccM$ with full measure.

\begin{theorem}\label{Th:Support2d}
	\BLUE{Fix $b\in (0,\kappa/2)$. }For any $\eps \in (0,1)$, $\supp(Q_\eps) \subset \supp(Q)$. Furthermore
	$$ \supp(Q) = \overline{\{(h,\vert \nabla G*h \vert^2 + c): h\in L^\infty_{p_b}(\R^2)\;,\; c\in\R\}}^{\ccM}\;.$$
\end{theorem}
\begin{remark}
	To identify the spectrum of $\cH$, we will only need the inclusion $\supp(Q_\eps) \subset \supp(Q)$. However, our proof of this inclusion relies on the identification of $\supp(Q)$.
\end{remark}

The rest of this subsection is devoted to the proof of Theorem \ref{Th:Support2d}: it relies on three lemmas. The first two lemmas are inspired by the work of Chouk and Friz~\cite{CF16}, while the third lemma is inspired by the work of Hairer and Sch\"onbauer~\cite{HS21}.\\

We first introduce, for any function $h \in L^{\infty}_{p_{\kappa/2}}(\R^2)$, the shift operator $T_h$ on $\ccM$ as follows:
\begin{equation} \label{eq:shift_on_Q}
	T_h q = T_h (X,U) = (X+h, U + 2 (\nabla G*X) \cdot (\nabla G*h) + \vert \nabla G*h\vert^2)\;.
\end{equation}
One can check that $(h,q) \mapsto T_h q$ is continuous from $L^{\infty}_{p_{\kappa/2}}(\R^2)\times \ccM$ into $\ccM$ and that $T_h^{-1} = T_{-h}$.

\begin{lemma}
	For all $\xi \in \Omega_0$ and all $h\in L^\infty_{p_{\kappa/2}}$, it holds $\xi+h \in \Omega_0$ and
	$$ T_h Q(\xi) = Q(\xi+h)\;.$$
\end{lemma}
\begin{proof}
	For any $\eps\in (0,1)$, any $\xi\in\Omega$ and all $h\in L^\infty_{p_{\BLUE{\kappa}/2}}$  it holds
	$$ T_{h_{\eps}} Q_\eps(\xi) = Q_\eps(\xi+h)\;$$
	where $h_\eps = h * \varrho_\eps$ converges to $h$ in $L^\infty_{p_{\kappa/2}}$ as $\eps \to 0$. We now restrict ourselves to $\xi \in \Omega_0$. We know that $Q_{\eps_k}(\xi)$ converges to $Q(\xi)$. The continuity of the shift operator thus implies that $T_{h_{\eps_k}} Q_{\eps_k}(\xi)$ converges to $T_h Q(\xi)$. We thus deduce that $Q_{\eps_k}(\xi+h)$ converges. To conclude the proof, it suffices to show that $\xi+h \in \Omega_0$ so that the limit of $Q_{\eps_k}(\xi+h)$ is necessarily equal to $Q(\xi+h)$ and therefore $T_h Q(\xi) = Q(\xi+h)$.\\
	To check that $\xi+h \in \Omega_0$, by the definition of $\Omega_0$ it suffices to show that for all $x\in \R^2$, $Q_{\eps_k}(\theta_x(\xi+h))$ converges. The continuity of $\theta_x$ in $\ccM$ combined with the convergence of $Q_{\eps_k}(\xi+h)$ and the identity $Q_{\eps_k}(\theta_x(\xi+h))=\theta_x Q_{\eps_k}(\xi+h)$ allows to conclude.
\end{proof}
Recall that $b\in (0,\kappa/2)$. 
\begin{lemma}\label{Lemma:ShifthQ}
	For any $q \in \supp(Q)$ and any $h\in L^\infty_{p_{b}}$, $T_h q \in \supp(Q)$.
\end{lemma}
\begin{proof}
	Suppose that the property holds true under the restrictive assumption that $h\in L^2\cap L^\infty_{p_{b}}$. Fix $h\in L^\infty_{p_{b}}$. There exists $h_n \in L^2\cap L^\infty_{p_{b}}$ such that $h_n \to h$ in $L^\infty_{p_{\kappa/2}}$ (for instance, take $h_n$ to be the restriction of $h$ to $[-n,n]^2$). The continuity of the shift operator then implies that $T_{h_n} q$ converges to $T_h q$. Since $\supp(Q)$ is a closed set, we conclude that $T_h q \in \supp(Q)$.\\
	Let us now prove that the property holds for $h\in L^2\cap L^\infty_{p_{b}}$. Assume that $q \in \supp(Q)$. For any open set $U\subset\ccM$ containing $q$, it holds $\P(Q\in U) > 0$. Fix some open set $V$ containing $T_h q$. By continuity, the set $U := T_h^{-1} V = T_{-h} V$ is open and contains $q$. Therefore $\P(Q\in V) = \P(Q \in T_h U) = \P(T_{-h} Q \in U)$. By the previous lemma, $T_{-h} Q(\xi) = Q(\xi-h)$ for all $\xi \in \Omega_0$. The Cameron-Martin Theorem ensures that the law of $\xi\mapsto Q(\xi-h)$ is equivalent to the law of $Q(\xi)$. Since $\P(Q\in U) > 0$, we deduce that $\P(T_{-h} Q\in U) > 0$ and consequently $\P(Q\in V) > 0$. We have therefore proven that any open set $V$ containing $T_h q$ has positive measure under the law of $Q$: necessarily $T_h q \in \supp(Q)$.
\end{proof}

\begin{lemma}\label{Lemma:Resonance}
	Fix $c\in \R$. There exist two functions $\delta\mapsto \lambda_\delta \in (0,1)$ and $\delta \mapsto a_\delta \in \R$ such that if for any $\xi\in\Omega$ we set
	$$ h_\delta(\xi) := -\xi_{\delta} + a_\delta \xi_{\lambda_\delta}\;,$$
	then the $\ccM$-valued random variable $T_{h_\delta} Q(\xi)$ converges in probability to $(0,c)$.
\end{lemma}
\begin{proof}
	Observe that, given any two functions $\delta\mapsto \lambda_\delta \in (0,1)$ and $\delta \mapsto a_\delta \in \R$, $h_\delta(\xi)$ belongs to $L^\infty_{p_b}$ so that $T_{h_\delta} Q(\xi)$ is well-defined. \BLUE{For any $\xi \in \Omega$, we write $Q(\xi) = (\xi, Z)$. Note that $T_{h_\delta} Q(\xi) = (\xi + h_\delta , Z + Y_\delta)$ where
	\begin{align*}
			Y_\delta= &|\nabla G * \xi_\delta|^2 - 2 a_\delta (\nabla G * \xi_\delta) \cdot (\nabla G * \xi_{\lambda_\delta}) + a_\delta^2 |\nabla G * \xi_{\lambda_\delta}|^2 \\&- 2 (\nabla G * \xi_\delta) \cdot (\nabla G * \xi) + 2a_\delta (\nabla G * \xi_{\lambda_\delta}) \cdot (\nabla G * \xi).
	\end{align*}
	We set $\lambda_\delta = \delta^{\delta^{-1}}$. Assume that there exists a function $\delta\mapsto a_\delta$ which is such that $a_\delta \to 0$ as $\delta \downarrow 0$ and
	\begin{equation}\label{Eq:Cdelta}
		C_\delta := \sup_{x\in\R^2} \sup_{\mu \in (0,1]} \sup_{\eta \in \cB^r} \frac{\E[\langle Z+Y_\delta-c,\eta^\mu_x\rangle^2]}{\mu^{-\kappa}}\to 0\;,\quad \delta \downarrow 0\;.
	\end{equation}
	Then we deduce that $\xi + h_\delta$ converges to $0$ in $\cC^{-\kappa}_{p_\kappa}$, and following a similar computation as in~\cite[Proof of Proposition 1.3]{HL15} one can obtain for $m$ large enough
	$$ \E\norm{Z+Y_\delta - c}_{\cC^{-\kappa}_{p_\kappa}}^{2m} \lesssim \sum_{x \in \Z^2} p_\kappa(x)^{-2m}\sum_{n \ge 0} 2^{2n} 2^{-nm\kappa} C_\delta^m \lesssim C_\delta^m\;,$$
	so the desired convergence of $Z+Y_\delta$ to $c$ follows.\\
	Let us now prove that there exists a function $\delta\mapsto a_\delta$ that goes to $0$ and such that~\eqref{Eq:Cdelta} holds. Since $Z+Y_\delta$ is a stationary field, it suffices to consider $x=0$ in~\eqref{Eq:Cdelta}. In fact, the Wiener chaos decomposition yields two terms
	\begin{equation}
		\langle Z+Y_\delta-c,\eta^\mu\rangle = E^{(0)}_{\delta, c} + E^{(2)}_{\delta, \mu}
	\end{equation}
	where the zeroth-order homogeneous Wiener chaos term $E^{(0)}_{\delta, c}$ is nothing but the expectation of $\langle Z+Y_\delta-c,\eta^\mu\rangle$, and $E^{(2)}_{\delta, \mu}$ denotes the remaining term belonging to the second-order homogeneous Wiener chaos. \eqref{Eq:Cdelta} will follow if we show:
	\begin{enumerate}
		\item provided that $a_\delta \to 0$ as $\delta \to 0$, one has
		\[\sup_{\mu \in (0,1]} \sup_{\eta \in \cB^r} \mu^\kappa\E[(E^{(2)}_{\delta, \mu})^2]\to 0, \quad \text{as } \delta \to 0; \]
		\item there exists a function $a_\delta$ and some $\delta_0 > 0$ such that $a_\delta \to 0$ as $\delta \downarrow 0$ and that $E^{(0)}_{\delta,c} = 0$ for all $\delta \in (0,\delta_0)$.
	\end{enumerate} 
	
	Let us first control the second-order component. In order to represent the stochastic integrals living in the Wiener chaos, let us introduce some graphical notations.
	Let the edges \tikz[baseline=-0.1cm] \draw[G] (0,0) to (1,0);
	and \tikz[baseline=-0.1cm] \draw[G] (0,0) to node[below, pos=0.5]{\tiny $\delta$} (1,0);
	represent the kernel $\nabla G(y-x)$ and $\nabla G_\delta(y-x)$, respectively, where in the argument $x, y$ denote the coordinates of the start and the end point of the edge.
	Similarly, the edge \tikz[baseline=-0.1cm] \draw[Gdiff] (0,0) to node[below, pos=0.5]{\tiny $\delta$} (1,0);
	represents the kernel $\nabla (G-G_\delta)(y-x)$.
	The edge \tikz[baseline=-0.1cm] \draw[testfcn] (0,0) to (1,0);
	represents the test function.
	Additionally, let the node \tikz[baseline=-0.1cm] \draw (0, 0) to node[root]{} (0,0);
	denote the origin $(0,0)$.
	The node \tikz[baseline=-0.1cm] \draw (0, 0) to node[xi]{} (0,0);
	denotes a variable to be integrated against the spatial white noise $\xi$, while \tikz[baseline=-0.1cm] \draw (0, 0) to node[dot]{} (0,0);
	denotes a node to be integrated out. As an example, we have the following identity $\iint \nabla G(x - y) \eta(x) dx \xi(dy) =
	\begin{tikzpicture}
		\node at (-1, 0) [xi] (left) {};
		\node at (0, 0)	[dot] (center) {};
		\node at (1, 0) [root] (right) {};
		\draw[G] (left) to (center);
		\draw[testfcn] (right) to (center);
	\end{tikzpicture}.$
	
	With these notations at hand, the second-order Wiener chaos component of $\langle Z+Y_\delta-c,\eta^\mu\rangle$ can be represented by
	\begin{equation}\label{eq:2nd-chaos}
		E^{(2)}_{\delta, \mu} = 
		\begin{tikzpicture}[scale=0.35,baseline=0.2cm]
			\node at (0,-1)  [root] (root) {};
			\node at (0,1)  [dot] (root2) {};
			\node at (-1.5,2.5)  [xi] (left) {};
			\node at (1.5,2.5)  [xi] (right) {};
			
			\draw[testfcn] (root) to  (root2);
			
			\draw[G] (left) to (root2);
			\draw[Gdiff] (right) to node[below, pos=0.4] {\tiny $\delta$} (root2);
		\end{tikzpicture}
		-
		\begin{tikzpicture}[scale=0.35,baseline=0.2cm]
			\node at (0,-1)  [root] (root) {};
			\node at (0,1)  [dot] (root2) {};
			\node at (-1.5,2.5)  [xi] (left) {};
			\node at (1.5,2.5)  [xi] (right) {};
			
			\draw[testfcn] (root) to  (root2);
			
			\draw[G] (left) to node[below, pos=0.4] {\tiny $\delta$} (root2);
			\draw[Gdiff] (right) to node[below, pos=0.4] {\tiny $\delta$} (root2);
		\end{tikzpicture}
		+2a_\delta
		\begin{tikzpicture}[scale=0.35,baseline=0.2cm]
			\node at (0,-1)  [root] (root) {};
			\node at (0,1)  [dot] (root2) {};
			\node at (-1.5,2.5)  [xi] (left) {};
			\node at (1.5,2.5)  [xi] (right) {};
			
			\draw[testfcn] (root) to  (root2);
			
			\draw[G] (left) to node[below, pos=0.4] {\tiny $\lambda$} (root2);
			\draw[Gdiff] (right) to node[below, pos=0.4] {\tiny $\delta$} (root2);
		\end{tikzpicture}
		+a_\delta^2
		\begin{tikzpicture}[scale=0.35,baseline=0.2cm]
			\node at (0,-1)  [root] (root) {};
			\node at (0,1)  [dot] (root2) {};
			\node at (-1.5,2.5)  [xi] (left) {};
			\node at (1.5,2.5)  [xi] (right) {};
			
			\draw[testfcn] (root) to  (root2);
			
			\draw[G] (left) to node[below, pos=0.4] {\tiny $\lambda$} (root2);
			\draw[G] (right) to node[below, pos=0.4] {\tiny $\lambda$} (root2);
		\end{tikzpicture}\;.
	\end{equation}
	In order to prove that the second moment of $E^{(2)}_{\delta, \mu}$ satisfies the desired property, it suffices to show that the second moment of each term in \eqref{eq:2nd-chaos} can be controlled by a quantity of order $A(\delta)\mu^{-\kappa}$ for some $A(\delta) \to 0$, uniformly over all $\mu$ and all test functions $\eta$. For this purpose, we shall use a well-known result concerning the generalized convolutions \cite[Theorem A.3]{HQ18} as well as the bounds
	\[|\nabla G(x)| + |\nabla G_\delta(x)| \lesssim |x|^{-1}, \quad |\nabla (G - G_\delta)(x)| \lesssim \delta^{\kappa'} |x|^{-1-\kappa'},\]
	that hold uniformly over $x \in \R^2$, $\delta \in (0, 1]$ and $\kappa' \in (0, 1)$. Indeed, for the first term in \eqref{eq:2nd-chaos}, one can check that
	\begin{equation}
		\E\left(
		\begin{tikzpicture}[scale=0.35,baseline=0.4cm]
			\node at (0,-1)  [root] (root) {};
			\node at (0,1)  [dot] (root2) {};
			\node at (-1.5,2.5)  [xi] (left) {};
			\node at (1.5,2.5)  [xi] (right) {};
			
			\draw[testfcn] (root) to  (root2);
			
			\draw[G] (left) to (root2);
			\draw[Gdiff] (right) to node[below, pos=0.4] {\tiny $\delta$} (root2);
		\end{tikzpicture}\right)^2
		=
		\begin{tikzpicture}[scale=0.35,baseline=0.4cm]
			\node at (0,-1)  [root] (root) {};
			\node at (0,3)	[dot]	(top)  {};
			\node at (0,1)  [dot] (bottom) {};
			\node at (-2,2)  [dot] (left) {};
			\node at (2,2)  [dot] (right) {};
			
			\draw[testfcn] (root) to  (left);
			\draw[testfcn] (root) to  (right);
			
			\draw[G] (top) to (left);
			\draw[G] (top) to (right);
			\draw[Gdiff] (bottom) to node[below, pos=0.4] {\tiny $\delta$} (left);
			\draw[Gdiff] (bottom) to node[below, pos=0.4] {\tiny $\delta$} (right);
		\end{tikzpicture}
		+
		\begin{tikzpicture}[scale=0.35,baseline=0.4cm]
			\node at (0,-1)  [root] (root) {};
			\node at (0,3)	[dot]	(top)  {};
			\node at (0,1)  [dot] (bottom) {};
			\node at (-2,2)  [dot] (left) {};
			\node at (2,2)  [dot] (right) {};
			
			\draw[testfcn] (root) to  (left);
			\draw[testfcn] (root) to  (right);
			
			\draw[G] (top) to (left);
			\draw[G] (bottom) to (right);
			\draw[Gdiff] (bottom) to node[below, pos=0.4] {\tiny $\delta$} (left);
			\draw[Gdiff] (top) to node[above, pos=0.4] {\tiny $\delta$} (right);
		\end{tikzpicture}
		\lesssim \delta^{2\kappa'} \mu^{-2\kappa'}
	\end{equation}
	by the procedure outlined in \cite{HP15}. Similarly, the second moment of the remaining three terms can be bounded by a quantity of order $\delta^{2\kappa'}\mu^{-2\kappa'}$, $a_\delta^2 \delta^{\kappa'} \mu^{-\kappa'}$ and $a_\delta^4 \mu^{-\kappa'}$, respectively. By choosing $2\kappa' < \kappa$, we thus conclude that the property 1. holds for $E^{(2)}_{\delta,\mu}$.}
	 
	\BLUE{Let us now turn to the property 2., i.e., prove that there exists a function $a_\delta$, that goes to $0$ as $\delta \downarrow 0$, and some $\delta_0 >0$ such that for all $\delta \in (0,\delta_0)$, the expectation of $\langle Z+Y_\delta-c,\eta^\mu\rangle$ vanishes. Let us first note that, by the renormalisation procedure in~\cite{HL15}, $Z$ lives in the homogeneous second order Wiener chaos so that its expectation vanishes. As a consequence, we only need to consider the expectation of $\langle Y_\delta-c,\eta^\mu\rangle$. Furthermore, $Y_\delta$ is function valued and stationary. We will thus show that there exists a function $a_\delta$, that goes to $0$ as $\delta \downarrow 0$, and some $\delta_0 >0$ such that for all $\delta \in (0,\delta_0)$, $\E[Y_\delta(0)] = c$.}\\
	For all constants $\delta_1, \delta_2 \in (0, 1)$, we introduce the notation
	\[v(\delta_1, \delta_2) := \E\left[(\nabla G * \xi_{\delta_1})(0) \cdot (\nabla G * \xi_{\delta_2})(0)\right] = \int_{\R^2} \nabla G_{\delta_1}(x) \cdot \nabla G_{\delta_2}(x) \dd x\;,\]
	and when either $\delta_1$ or $\delta_2$ equals $0$, we replace $\nabla G_{\delta_j}$ simply by $\nabla G$ in the expression. It can be checked that there exist $0 < c_1 < c_2$ such that for all $\delta_1,\delta_2$ as above
	\[ c_1 \log[(\delta_1 \vee \delta_2)^{-1}] \le v(\delta_1, \delta_2) \le c_2 \log[(\delta_1 \vee \delta_2)^{-1}] \;.\]
	The expectation of \BLUE{$Y_\delta(0)-c$ then equals
	$$v(\delta, \delta) - 2a_\delta v(\delta, \lambda_\delta) + a_\delta^2 v(\lambda_\delta, \lambda_\delta) - 2v(\delta, 0) + 2 a_\delta v(\lambda, 0) - c\;.$$}
	Recall that $\lambda_\delta = \delta^{\delta^{-1}}$. The proof is complete if we can show there exist two parameters $\delta_0\in (0,1)$ and $R_0 > 0$, as well as a function $a_\delta$ satisfying $|a_\delta| \leq R_0 \delta$ and
	\[c = v(\delta, \delta) - 2a_\delta v(\delta, \lambda_\delta) + a_\delta^2 v(\lambda_\delta, \lambda_\delta) - 2v(\delta, 0) + 2 a_\delta v(\lambda_\delta, 0)\;,\quad \forall \delta \in(0,\delta_0)\;.\]
	We proceed through a fixed point argument. For $\delta_0 \in (0, 1)$, define $F_{\delta_0}$ as the metric space of all continuous functions $f:(0,\delta_0) \to\R$ such that
	\[ \norm{f}_{\delta_0} = \sup_{\delta \in (0, \delta_0)} \frac{|f(\delta)|}{\delta} < \infty\;.\]
	For $R>0$, let $F_{\delta_0,R}$ be the subset of all $f$ such that $\norm{f}_{\delta_0} \le R$. For any $f\in F_{\delta_0}$ define
	\[Mf(\delta) := \frac{1}{2v(\lambda_\delta, 0)} \left(c - v(\delta, \delta) + 2f(\delta) v(\delta, \lambda_\delta) - f(\delta)^2 v(\lambda_\delta, \lambda_\delta) + 2v(\delta, 0)\right)\]
	Since $\log(\lambda^{-1}_\delta) = \delta^{-1} \log(\delta^{-1})$, one can check that, for all $\delta_0 \in (0,1)$, $M$ maps $F_{\delta_0}$ into itself. Furthermore, there exists $R_0>0$ such that for all $\delta_0$ small enough, $M$ maps $F_{\delta_0,R_0}$ into itself.\\
	Moreover, there exists a constant $C>0$ such that for any two elements $f, g \in F_{\delta_0,R_0}$, we have
	\begin{align*}
		\norm{Mf - Mg}_{\delta_0} &\leq \norm{f - g}_{\delta_0} \sup_{\delta \in (0, \delta_0)} \left| \frac{2 \vert v(\delta, \lambda_\delta) \vert + |f(\delta)+g(\delta)| \vert v(\lambda_\delta, \lambda_\delta)\vert}{2v(\lambda, 0)} \right|\\
		&\le C\delta_0 (1+R_0)  \norm{f - g}_{\delta_0} \;.
	\end{align*}
	Choose $\delta_0$ such that $C\delta_0 (1+R_0) < 1$. Then $M$ admits a unique fixed point $a\in F_{\delta_0,R_0}$. \BLUE{One can then define $a_\delta$ arbitrarily on $(\delta_0,1)$ and this} suffices to conclude the proof.
\end{proof}

\begin{proof}[Proof of Theorem \ref{Th:Support2d}]
	Set $A:= \overline{\{(h,\vert \nabla G*h \vert^2 + c): h\in L^\infty_{p_b}(\R^2)\;,\; c\in\R\}}^{\ccM}$. For all $\eps \in (0,1)$, the r.v.~$Q_\eps$ takes values in the closed set $A$: consequently $\supp(Q_\eps) \subset A$. Since $Q$ is the limit in probability of $Q_\eps$, we deduce that $\supp(Q) \subset A$. It remains to prove the converse inclusion.\\
	Fix $c\in \R$. Assume that $(0,c) \in\supp(Q)$. Then Lemma \ref{Lemma:ShifthQ} implies that for all $h \in L^\infty_{p_b}$, $T_h (0,c) \in \supp(Q)$ and the desired inclusion follows. It remains to show that $(0,c) \in \supp(Q)$.\\
	Using the notations of Lemma \ref{Lemma:Resonance}, for any $\delta \in (0,1)$, we set
	$$ h_\delta(\xi) := -\xi_{\delta} + a_{\delta} \xi_{\lambda_\delta}\;,$$
	 By this lemma, there exists $\delta_n \downarrow 0$ such that on some event $\Omega_0'$ of full $\P$-measure, $T_{h_{\delta_n}(\xi)} Q(\xi)$ converges to $(0,c)$ for all $\xi\in\Omega_0'$.\\
	Fix $q \in \supp(Q) \cap Q(\Omega_0')$. There exists $\xi \in \Omega_0'$ such that $q = Q(\xi)$. By Lemma \ref{Lemma:ShifthQ}, $T_{h_{\delta_n}(\xi)} Q(\xi) \in \supp(Q)$. Since the support is a closed set, we deduce that $(0,c) \in \supp(Q)$.
\end{proof}

\section{The Anderson Hamiltonian in dimension three}\label{sec:3d}

We follow the same strategy as in dimension two. However, the construction of the (PAM) in dimension three
\begin{equation}\label{eq:PAM3d}
	\begin{cases}
		\partial_t u = \Delta u - \xi u,&  t >0,~ x \in \R^3,\\
		u(t = 0, \cdot) = f
	\end{cases}
\end{equation}
is much harder as the noise is more singular, we thus rely on~\cite{HL18} which performed its construction with the theory of regularity structures~\cite{Hai14}.\\
In this whole section, $\kappa > 0$ is a small parameter (it is implicitly taken as small as needed). Furthermore we set
$$ \gamma := \frac32 + 2\kappa\;,\quad \alpha := -\frac32-\kappa\;.$$

\subsection{The space of enhanced noises and the white noise case}

A regularity structure is a triplet $(\cA, \cT, \cG)$ satisfying the following properties:
	\begin{enumerate}
		\item $\cA \subset \R$ is a locally finite set of indices that is bounded from below and contains $0$.
		\item $\cT = \bigoplus_{\alpha \in \cA} \cT_\alpha$ is a graded vector space, where for each $\alpha \in \cA$, $\cT_\alpha$ is a finite-dimensional Banach space equipped with the norm $\norm{\cdot}_{\alpha}$. We impose $\cT_0 \simeq \R$ with the unit vector $\un$. For any vector $\tau$ in a finite-dimensional subspace of $\cT$, we write $\norm{\tau}$ for its Euclidean norm. 
		\item $\cG$ is a group acting on $\cT$ such that every element $\Gamma$ of $\cG$ satisfies $\Gamma \un = \un$ and for all $\tau \in \cT_{\alpha}$, $\Gamma\tau - \tau \in \cT_{<\alpha} := \bigoplus_{\beta \in \cA_{<\alpha}} \cT_{\beta}$ where $\cA_{<\alpha} = \cA \cap (-\infty, \alpha)$.
	\end{enumerate}

A basic example is the polynomial regularity structure where $\cA = \N$, and for each $n\in\N$, $\cT_n$ is the vector space spanned by all $X^k:=X_1^{k_1}X_2^{k_2}X_3^{k_3}$, $k\in \N^3$ such that $|k|:=k_1+k_2+k_3= n$ and $\cG$ is the group formed by all the transformations $\Gamma_h$ that translate any polynomial $P$ by the vector $h$: $\Gamma_h P(X) = P(X+h)$.

Let us now introduce the regularity structure associated to the (PAM). Let $\Xi$ be a formal expression (aimed at representing $\xi$ at the level of the model space $\cT$). Let $\cF$ and $\cU$ be the smallest sets of formal expressions such that $\cF$ contains $\Xi$, $\cU$ contains all polynomials $X^k$ and
\[\tau \in \cU \implies \tau \Xi \in \cF ~~\text{and}~~ \tau \in \cF \implies \cI(\tau) \in \cU.\]
Here $\tau\Xi$ and $\cI(\tau)$ denote new formal expressions that need to be considered. The formal expressions in $\cU$ will be used in the local description of the solution $u$ while those in $\cF$ will be used for the product $u\xi$.

Each expression $\tau$ is assigned a number $|\tau|$ called homogeneity, which is calculated by the following rules: (1) $|X^k| = |k|$, (2) $|\Xi| = \alpha$, (3) for any $\tau, \tau'$, $|\tau\tau'| = |\tau| + |\tau'|$, (4) for any $\tau$, $|\cI(\tau)| = |\tau| + 2$.\\
We thus set $\cA := \{|\tau|: \tau \in \cF \cup \cU\}$ and for $\alpha \in \cA$, we let $\cT_{\alpha}$ be the vector space spanned by all $\tau \in \cF \cup \cU$ such that $|\tau| = \alpha$. For the construction of the structure group $\cG$, we refer the reader to \cite[Sec. 8.1]{Hai14}.

From now on, we will \emph{always} restrict $\cU$ (resp.~$\cF$) to formal expressions whose homogeneities are below $\gamma$ (resp.~$\gamma +\alpha$), see Figure \ref{Table}.

\begin{figure}
	\begin{center}
		\begin{tabular}{p{2.1cm} l  | p{2.4cm} l  }\hline
			$\cU$ & $\cA(\cU)$  & $\cF$ & $\cA(\cF)$ \\
			\hline
			$\mathbf{1}$ &  $0$ & $\Xi$ &  $-\frac{3}{2}-\kappa$ \\
			$\cI(\Xi)$  & $\frac{1}{2}-\kappa$ & $\Xi\cI(\Xi)$ & $-1-2\kappa$ \\
			$\cI(\Xi\cI(\Xi))$ & $1-2\kappa$  & $\Xi\cI(\Xi\cI(\Xi))$ & $-\frac{1}{2}-3\kappa$ \\
			$X_i$ & $1$ & $\Xi X_i$ & $-\frac{1}{2}-\kappa$\\
			$\cI(\Xi\cI(\Xi\cI(\Xi)))$ & $\frac{3}{2}-3\kappa$ & $\Xi\cI(\Xi\cI(\Xi\cI(\Xi)))$ & $-4\kappa$ \\
			$\cI(\Xi X_i)$ & $\frac{3}{2}-\kappa$ & $\Xi\cI(\Xi X_i)$ & $-2\kappa$\\
			\hline
	\end{tabular}\end{center}
	\caption{Basis vectors and homogeneities in the regularity structure.}\label{Table}
\end{figure}

We also consider the notion of admissible models as introduced in~\cite[Definition 2.2]{HL18} associated to a compactly supported kernel $\bar{K}_+$ which coincides near $0$ with the Green's function of the Laplacian in dimension $3$. We will denote by $q=(\Pi,\Gamma)$ a generic admissible model.
\begin{remark}
	In~\cite{HL18}, admissible models are in space-time and are taken w.r.t.~the kernel $P_+$, which is a space-time function that coincides with the heat kernel near the origin: the setting therein was designed to encompass simultaneously the multiplicative (SHE) in dimension $1$ (for which the noise is space-time) and the (PAM) for which the noise is in space only. In the latter case, that we consider in this article, the stochastic objects only evolve in space so that the model is in space only and the kernel $\bar{K}_+$ is nothing but the integral in time of $P_+$.
\end{remark}
We recall the notation
\begin{align*}
	\| \Pi  \|_x := \sup_{\varphi\in\cB_r}\sup_{\lambda \in (0,1]} \sup_{\substack{\zeta\in\cA\\ \tau \in \cT_\zeta}} \frac{\vert\langle \Pi_x \tau  , \varphi^\lambda_x\rangle\vert}{\lambda^{\zeta} \|\tau\| }\;,\;\| \Gamma  \|_{x,y} := \sup_{\substack{\beta \le \zeta \in \cA\\ \tau \in \cT_\zeta}} \frac{\vert \Gamma_{x,y} \tau \vert_{\beta}}{|x-y|^{\zeta-\beta}\|\tau\|}\;.
\end{align*}
The set of all admissible models $q=(\Pi,\Gamma)$ is endowed with the distance
$$ \| q;\bar{q}\| := \sup_{x\in \R^3} \frac{\| \Pi - \bar{\Pi} \|_x}{p_\kappa(x)} + \sup_{x,y\in \R^3:|x-y|\le 1} \frac{\| \Gamma - \bar{\Gamma} \|_{x,y}}{p_\kappa(x)}\;.$$
It turns out that for admissible models the $\Gamma$ can be read off the $\Pi$: more precisely
\begin{equation}\label{def:normm}
	\normm{q;\bar{q}} := \sup_{x\in\R^3} \sup_{\varphi\in\cB_r}\sup_{\lambda \in (0,1]} \BLUE{\sup_{\substack{\zeta\in\cA\\ \tau \in \cT_\zeta}}} \frac{\vert \langle \Pi_x \tau - \bar{\Pi}_x\tau , \varphi^\lambda_x\rangle \vert}{\lambda^{|\tau|} p_\kappa(x)\BLUE{\|\tau\|}}\;,
\end{equation}
defines an equivalent metric on the set of admissible models, see for instance~\cite[Sec 2.4]{HW14}.\\
An admissible model $q=(\Pi,\Gamma)$ is called \emph{smooth} if $\Pi_x\tau \in C^\infty(\R^3)$ for all $x\in\R^3$ and all $\tau\in\cA_{}$. We then let $\ccM$ be the closure of all smooth models in the set of admissible models, it is a complete and separable metric space.\\

We fix some parameter $b \in (0,\kappa/4)$. Given some function $h\in C^\infty_{p_{b}}(\R^3)$, one can define the canonical model $q_{\mbox{\tiny can}}(h)$ associated to $h$: this is the only admissible model satisfying
$$\Pi^{\mbox{\tiny can}}(h)_x(\Xi)(y) = h(y) \mbox{ and }\Pi^{\mbox{\tiny can}}(h)_x (\tau\bar\tau)(y) = \Pi^{\mbox{\tiny can}}(h)_x (\tau)(y) \Pi^{\mbox{\tiny can}}(h)_x (\bar\tau)(y)\;.$$
Let us mention that we took $b$ smaller than $\kappa/4$ in order for the norm of this model to be finite: indeed $4$ is the largest number of occurrences of $\Xi$ in the basis vectors of the regularity structure at stake and $p_b^4 \le p_\kappa$.\\
\BLUE{In order to implement some renormalisation procedure, it is necessary to introduce the notion of \emph{renormalisation group}. In the case of the generalised (PAM), it is a set of maps $M$ acting on a subspace $\cT_0$ of the regularity structure $\cT$, that satisfy several algebraic conditions listed in~\cite[Sec 4.2]{HP15} and which allow to associate to any model $(\Pi,\Gamma)$ a renormalised model $(\hat{\Pi}, \hat{\Gamma})$: we refer to~\cite[Eq (4.4)]{HP15} for the expression of the model.\\
In~\cite[Section 4.3]{HP15}, a three dimensional renormalisation subgroup was identified in the context of the generalised (PAM) and this subgroup was sufficient to renormalise the model associated to a white noise. In the present article, we only deal with the linear (PAM) so that a two dimensional renormalisation subgroup suffices. This two dimensional subgroup is parametrized by a pair $\bar{c} = (c,c^{(1)}) \in \R^2$ and the map $M$ writes $\exp(-cL - c^{(1)} L^{(1)})$ for some linear maps $L$ and $L^{(1)}$ provided in~\cite[Section 4.3]{HP15}. Given a canonical model $q_{\mbox{\tiny can}}(h)$, this produces a \emph{renormalised} canonical model $\ccR^{\bar c} q_{\mbox{\tiny can}}(h) = (\hat{\Pi}(h), \hat{\Gamma}(h))\in\ccM$ that satisfies
	\[ \hat{\Pi}(h)_x \Xi \cI(\Xi)(x) = - c\;,\quad \hat{\Pi}(h)_x \Xi \cI(\Xi \cI(\Xi \cI(\Xi)))(x) = - c^{(1)}\;,\quad x\in\R^3\;.\]}

We set $\Omega := \cC^{-\frac32-\kappa}_{p_b}(\R^3)$ endowed with the law $\P$ of white noise. The canonical variable on $\Omega$ will be denoted $\xi$. Following~\cite[Theorem 5.3]{HL18}, we consider the {renormalised} canonical model $Q_\eps(\xi) = \ccR^{\bar c_\eps} q_{\mbox{\tiny can}}(\xi_\eps)$ associated to the smooth noise $\xi_\eps = \xi * \varrho_\eps$ and \BLUE{a well-chosen pair of renormalisation constants $(c_\eps, c_\eps^{(1)})$. Let us mention that the renormalisation constants satisfy as $\eps\downarrow 0$
	$$c_\eps = c_\varrho \eps^{-1} + O(1)\;,\quad c_\eps^{(1)} = \ln \eps^{-1} + O(1)\;,$$
	where $c_\varrho \in \R$ only depends on the function $\varrho$.}
The r.v.~$Q_\eps$ converges in $\ccM$ to some limit $Q$ in probability, see~\cite[Theorem 5.3]{HL18}. Actually, in finite volume, this result was proven by Hairer and Pardoux~\cite{HP15}, its extension to infinite volume rests on the same arguments as in \eqref{Eq:xieps2d} with $\xi_\eps$ ``replaced'' by $Q-Q_\eps$: all the stochastic objects live in finite Wiener chaoses and therefore admit finite moments.\\

To identify the spectrum of $\cH$ in Section \ref{sec:spectrum}, we will need to show a commutation property: the operator $\cH(\theta_x \xi)$ associated to the shifted noise coincides with the operator $\cH(\xi)$ conjugated with spatial shifts, see Lemma \ref{lem:shiftshift}. One therefore needs to construct the shifted model for \emph{all} $x\in\R^3$ simultaneously. This is the purpose of the next result. \BLUE{Before that, let us introduce the following notation. For any $q = (\Pi,\Gamma)\in\ccM$ and any $x\in\R^d$, we set $\theta_x q =(\theta_x \Pi,\theta_x \Gamma)$ where for any $y,z\in\R^3$
	$$ (\theta_x \Pi_y) \tau := \theta_x (\Pi_{y-x} \tau)\;,\quad \theta_x \Gamma_{y,z} := \Gamma_{y-x,z-x}\;.$$
Note that these identities imply that for any $\varphi \in \cB_r$
$$ \langle (\theta_x \Pi_y) \tau , \varphi\rangle = \langle \Pi_{y-x} \tau , \varphi(\cdot + x)\rangle\;.$$
}

\begin{lemma}\label{lem:Omega03d}
	There exists a sequence $\eps_k \downarrow 0$ such that the set
\begin{equation}\label{Eq:Omega0_3d}
	\Omega_0 : =\Big\{\xi\in\Omega: Q_{\eps_k}(\theta_x \xi) \mbox{ converges in $\ccM$ as } k\to\infty\;,\quad \forall x\in \R^3\Big\}\;,
\end{equation}
is of full $\P$-measure and is invariant under all $\theta_x, x\in \R^3$. The\footnote{$Q$ is defined as a limit on the set of full measure $\Omega_0$, and is extended to $\Omega$ arbitrarily.} limit $Q$ satisfies $Q(\theta_x \xi) = \theta_x Q(\xi)$ for all $\xi \in\Omega_0$.\\
As $\eps\downarrow 0$, the field $(Q_\eps(\theta_x \xi))_{x\in \R^3}$ which takes values in $\ccM$ converges in probability, locally uniformly over $x\in \R^3$, to $(Q(\theta_x \xi))_{x\in \R^3}$.
\end{lemma}
\begin{proof}
	\BLUE{We write $Q_\eps =: (\Pi^{(\eps)}, \Gamma^{(\eps)})$. }It is straightforward to check that for any $\eps\in (0,1)$, any $x\in\R^3$ and any $\xi\in\Omega$\BLUE{
	$$ Q_\eps(\theta_x \xi) = \theta_x Q_\eps(\xi)\;.$$}
	This observation combined with the definition of the metric \eqref{def:normm} leads to
	$$ \normm{Q_\eps(\theta_x \xi); Q_{\eps'}(\theta_x \xi)} \le \sup_{y\in \R^3} \frac{p_\kappa(y-x)}{p_\kappa(y)}\; \normm{Q_\eps(\xi); Q_{\eps'}(\xi)} = p_\kappa(x) \normm{Q_\eps(\xi); Q_{\eps'}(\xi)}\;,$$
	for all $\xi\in\Omega$ and all $x\in\R^d$. The arguments in~\cite[Theorem 5.3]{HL18} show that, for any given $p>1$, $Q_\eps(\xi)$ converges in $L^p(\Omega,\P)$ as $\eps\downarrow 0$ to some limit $Q(\xi)$ in $\ccM$. We thus deduce that the collection of r.v.~$(Q_\eps(\theta_x \xi),x\in\R^3)$, taking values in $L^\infty_{p_\kappa}(\R^3,\ccM)$, converges as $\eps\downarrow 0$ in $L^p(\Omega,\P)$. Therefore there exists a deterministic sequence $\eps_k \downarrow 0$ such that the convergence holds almost surely. We thus consider the set of full $\P$-measure
	$$ \Omega_0 := \Big\{\xi\in\Omega: Q_{\eps_k}(\theta_x \xi) \mbox{ converges in $\ccM$ as }k\to\infty\;,\quad \forall x\in \R^2\Big\}\;.$$
	We can conclude the proof using the exact same arguments as in dimension $2$, see the proof of Lemma \ref{lem:Omega02d}.
\end{proof}

\subsection{The parabolic evolution}

\begin{theorem}\label{Th:PAM3d}
	There exists a map $(q,f,t)\mapsto u^{q,f}(t,\cdot)$ defined on $\ccM\times \Big( \bigcup_{\ell_0\in\R} L^2_{e_{\ell_0}}\Big) \times \R_+$ such that for any $\ell_0\in\R$ the following properties hold:
	\begin{enumerate}
		\item (Continuity in the data) For any $t>0$, the map $(q, f) \mapsto u^{q, f}(t,\cdot)$ is continuous from $\ccM \times L^2_{e_{\ell_0}}$ into $L^2_{e_{\ell_0+t}}$, and is linear w.r.t.~$f$.
		\item (Continuity in time) For any $(q,f) \in \ccM\times L^2_{e_{\ell_0}}$ and any $T>0$, the map $t \mapsto u^{q, f}(t, \cdot)$ is continuous from $[0,T]$ to $L^2_{e_{\ell_0 + T}}$.
		\item (Semigroup property) For any $(q,f) \in \ccM\times L^2_{e_{\ell_0}}$ and any $0\le s \le t$, $u^{q, f}(t, \cdot)$ coincides with $u^{q,g}(s,\cdot)$ where $g = u^{q,f}(t-s,\cdot)$.
		\item (Symmetry) For any $q\in \ccM$, $f,g \in C_c$ and any $t\ge 0$, it holds
		$$ \langle f, u^{q,g}(t)\rangle =  \langle u^{q,f}(t) , g\rangle\;.$$
		\item (Classical solution) When $q = Q_\eps(\xi)$ for some $\xi\in\Omega$ and $f\in L^2_{e_{\ell_0}}$, $u^{q, f}$ coincides with the classical solution to \eqref{eq:PAM3d} with noise term $\xi_\eps + C_\eps$ \BLUE{where $C_\eps := c_\eps + c_\eps^{(1)}$.}
	\end{enumerate}
\end{theorem}
	
	\BLUE{The rest of this subsection is devoted to the proof of this theorem. For simplicity, we work with $T\in (0,1)$ as it simplifies some notations (but the proof can be adapted to cover a general $T$) and we fix some $\ell_0\in\R$. We use some material from~\cite{HL18}, that we need to introduce. Recall that $\kappa>0$ is a small parameter, that $\gamma = \frac32 + 2\kappa$ and $\alpha = -\frac32 - \kappa$. We set $\eta = - \frac12 + 3\kappa$, this parameter will control the worst regularity allowed for the initial condition in the solution theory of the PAM (we do not need such a bad regularity in the present context). Given a model $q\in\ccM$, we introduce the space of \emph{modelled distributions} $\cD^{\gamma,\eta,2}_{T,w}(q)$ as the set of all maps $\textrm{h}:(0,T]\times \R^3\to \cT_{<\gamma}$ satisfying
	\begin{align}
		&\sup_{t\in (0,T]} \bigg\| \frac{\vert \textrm{h}(t,x) \vert_\zeta}{t^{\frac{(\eta-\zeta)\wedge 0}{2}} w^{(1)}_t(x,\zeta)} \bigg\|_{L^2} \\
		&\sup_{\lambda \in (0,2]}\sup_{t\in (2\lambda^2,T]} \bigg\|\int_{y\in B(x,\lambda)} \lambda^{-3} \frac{\vert \textrm{h}(t,y) - \Gamma_{y,x} \textrm{h}(t,x) \vert_\zeta}{t^{\frac{\eta-\gamma}{2}} w^{(2)}_t(x,\zeta) \lambda^{\gamma-\zeta}} dy \bigg\|_{L^2} \label{Eq:hSpaceReg}\\
		&\sup_{t\in (0,T]}\sup_{s \in (t/2,t)} \bigg\|\frac{\vert \textrm{h}(t,x) - \textrm{h}(s,x) \vert_\zeta}{t^{\frac{\eta-\gamma}{2}} w^{(1)}_t(x,\zeta) \vert t-s\vert^{\frac{\gamma-\zeta}{2}}} \bigg\|_{L^2} \label{Eq:hTimeReg}\;.
	\end{align}
	Here the functions $w^{(i)}_t(x,\zeta)$ are weights, their precise expressions are provided in~\cite[Eq (4.1)]{HL18} but are unimportant for this proof so we do not present them: let us simply mention that they all coincide with $e_{\ell_0+t}(x)$ up to some factors $p_{a}(x)$, for some constants $a$ that depend on $\zeta$.\\
	
	The general idea of the theory of regularity structures consists in lifting the stochastic PDE at stake at the abstract level of modelled distributions. In the present setting this means that, instead of identifying a fixed point in some space of functions or distributions to the mild formulation of the PAM:
	$$ u = K*(-u\xi) + K*f\;,$$
	one identifies a fixed point, in some space of modelled distributions given some model $q$, to the equation:
	\begin{equation}\label{Eq:FixedPt3d}
		\textrm{u} = (\cP_+ + \cP_-) ( -\textrm{u} \Xi) + \cP f\;.
	\end{equation}
	Here $\cP f$ is a modelled distribution that stands for $K*f$, while $\cP_++\cP_-$ is the lifted space-time convolution operator associated to the heat kernel $K$. Actually, in the theory of regularity structures, $K$ is split into a local part $K_+$ which admits a singularity at the origin, and a global part $K_-$ which is smooth: the operators $\cP_+$ and $\cP_-$ are associated respectively to each of these two terms. Since we do not need details on these two operators, we do not recall their expressions.\\
	It is proven in~\cite[Theorem 5.2]{HL18} that there is a unique solution $\textrm{u}\in\ccD^{\gamma,\eta,2}_{T,w}(q)$ to the fixed point problem~\eqref{Eq:FixedPt3d} and that this solution is continuous w.r.t.~the pair $(q,f)$.\\
	This solution $\textrm{u}$ is ``abstract'' unless we can map it to genuine functions or distributions. This is the purpose of the reconstruction operator $\cR$, which is a cornerstone of the theory. In Proposition \ref{Prop:Recons} we will construct a particular instance of this operator. For the moment, let us simply mention that it interacts nicely with the convolution operator, indeed by~\cite[Theorem 4.3 and Lemma 4.6]{HL18}, it holds
	$$ \cR \Big((\cP_+ + \cP_-) ( \textrm{u} \Xi)\Big) = K* \cR (\textrm{u} \Xi)\;,$$
	as well as $\cR (\cP f) = Kf$. Consequently the fixed point $\textrm{u}$ satisfies
	$$ \cR  (\textrm{u}) = K* \cR (-\textrm{u} \Xi) + Kf\;.$$
	To complete the picture, let us mention that if ones takes $q = Q_\eps$ then it holds that
	$$ \cR (\textrm{u}) = K* \Big((\cR \textrm{u}) (-\xi_\eps - C_\eps)\Big) + Kf\;,$$
	so that $\cR (\textrm{u})$ coincides with the classical solution $u_\eps$ of
	$$ \partial_t u_\eps = \Delta u_\eps + u_\eps(-\xi_\eps - C_\eps)\;.$$
	The continuity of the solution map and of the reconstruction operator shows that the sequence $u_\eps = \cR (\textrm{u})$ converges as $\eps\downarrow 0$ to a limit $u$ which is the reconstruction $\cR (\textrm{u})$ of the solution $\textrm{u}$ associated to the renormalised model $q=Q$, the latter being the limit of $Q_\eps$ as $\eps\downarrow 0$.\\
	
	Let us now come back to Theorem \ref{Th:PAM3d}. Most of the ingredients were established in~\cite{HL18}. To complete the proof, we need to argue} that $\cR (\textrm{u} \Xi)$ can be seen as a continuous function of time valued in a space of distributions on $\R^3$: this will allow us to write the convolution in space-time as an integral in time of a convolution in space, see~\eqref{Eq:Mild3d}, and then to apply the very same arguments as in dimension $2$ to deduce the properties listed in the statement of the theorem.\\
	
	For $\beta < 0$, $\ell\in\R$ and $\delta > 0$, let $\cB^{\beta}_{\ell,\delta}$ be the Banach space of all distributions $g$ on $\R^3$ such that
	\begin{equation}
		\norm{g}_{\cB^{\beta}_{\ell,\delta}} := \sup_{\lambda \in (0, 1]} \norm{ \sup_{\varphi \in \cB^r(\R^3)} \frac{|\crochet{g, \varphi^\lambda_{x}}|}{e_\ell(x) p_\delta(x)\lambda^\beta}}_{L^2(\R^3, \dd x)} < \infty\;.
	\end{equation}
	This space is nothing but a weighted Besov space with parameters $p=2$ and $q=\infty$. Given $\nu\in\R$ and $T>0$, let $\ccE^{\beta,\nu}_{\ell_0,\delta, T}$ be the Banach space of all functions $g$ on $(0, T)$ valued in the space of distributions on $\R^3$ such that
	$$ \| g\|_{\ccE^{\beta,\nu}_{\ell_0,\delta, T}} := \sup_{t\in (0,T)} \frac{\| g(t)\|_{\cB^{\beta}_{\ell_0+t,\delta}}}{t^{\frac{\nu}{2}}} + \sup_{t\in (0,T): s\in (t/2,t)} \frac{\| g(t)-g(s)\|_{\cB^{\beta}_{\ell_0+t,\delta}}}{t^{\frac{\nu}{2}}|t-s|^{\kappa/4}}<\infty\;.$$
	\BLUE{Let us mention that the parameter $\nu$ controls the blow-up of the norm at time $0$.\\
		In order to establish Theorem \ref{Th:PAM3d}, we will prove the following
	\begin{proposition}\label{Prop:ContinuityInTime}
	There exists some constant $\delta > 0$ such that
	$$ (q,f) \mapsto \cR(\textrm{u} \Xi)\;,$$
	is a continuous map from $\ccM \times L^2_{e_{\ell_0}}$ into $\ccE^{\alpha,\eta+\alpha-\kappa/2}_{\ell_0,\delta, T}$.
	\end{proposition}
	Let us postpone the proof of this proposition and carry out the proof of the theorem.}
	\begin{proof}[Proof of Theorem \ref{Th:PAM3d}]
	Using Proposition \ref{Prop:ContinuityInTime}, we deduce that
	\begin{equation}\label{Eq:Mild3d}
		u^{q,f}(t) := \cR(\textrm{u})(t) = \int_0^t K(t-s)* \cR (-\textrm{u} \Xi)(s) ds+ K(t)*f\;,\quad t \in (0,T)\;.
	\end{equation}
	Then the very same arguments as in dimension $2$ allow to deduce Properties 1., 2.~and 3. On the other hand, Property 5.~was already proven in~\cite[Theorem 5.3]{HL18}. Regarding Property 4., the continuity properties already proven show that one can restrict oneself to a smooth model $q$. In this context, there exists a smooth function $\zeta$ that only depends on the model $q$ which is such that
	$$ \cR (\textrm{u} \Xi)(t,x) = u(t,x) \zeta(x)\;.$$
	If we let $u^f(t,x), u^g(t,x)$ be the solutions starting from $f,g$ then these two functions are strong solutions to the PDE
	$$ \partial_t u = \Delta u - u \zeta\;,$$
	starting from $f$ and $g$ respectively. Now if we set
	$$ h(s) := \langle u^f(t-s) , u^g(s) \rangle\;,\quad s\in [0,t]\;,$$
	then the desired property can be written as $h(0) = h(t)$. However $h$ is a continuous function and an integration by parts shows that $h'(s) = 0$ for all $s\in (0,t)$.
	\end{proof}
	
	We are left with the proof of Proposition \ref{Prop:ContinuityInTime}. 	\BLUE{As pointed out above, the bulk of the work consists in proving that the reconstruction of $\textrm{f} := \textrm{u}\Xi$ can be seen as a function of time taking values in a space of distributions of the space variable. This requires some work at the level of the reconstruction operator. Let us recall that a reconstruction operator $\cR$ acting on space-time modelled distributions was constructed in~\cite[Thm 3.10]{HL18}. In the proposition below, we will define a reconstruction operator that acts on modelled distributions in space only. Actually, for any given $t>0$ we will define a reconstruction operator $\cR_t$ acting on some space of modelled distributions whose norm depends on $t$. Later on, we will apply $\cR_t$ to the $t$-marginal of $\textrm{f}$. Of course, the family $\cR_t$, $t >0$ is consistent in the sense that $\cR_t$ and $\cR_s$ match on the intersection of the spaces of modelled distributions on which they act.\\}
	Fix $\tilde{\gamma} > 0$, $\tilde{\eta} < \alpha$, $\delta > 0$ and $t>0$. Given a model $q$, we let $\ccD^{\tilde{\gamma},\tilde{\eta}}_{\ell_0,\delta,t}(q)$ be the set of maps $\textrm{h}:\R^3 \to \cT$ satisfying
	\begin{align*}
		\|\textrm{h}\|_{\tilde{\gamma},\tilde{\eta},\ell_0,\delta,t} &:= \sum_{\zeta \in \cA_{<\tilde\gamma}}\Big( \bigg\| \frac{\vert \textrm{h}(x) \vert_\zeta}{t^{\frac{\tilde\eta-\zeta}{2}} e_{\ell_0+t}(x)p_\delta(x)} \bigg\|_{L^2}\\
		&+ \sup_{\lambda \in (0,\sqrt{t/2}]} \bigg\|\int_{y\in B(x,\lambda)} \lambda^{-3} \frac{\vert \textrm{h}(y) - \Gamma_{y,x} \textrm{h}(x) \vert_\zeta}{t^{\frac{\tilde\eta-\tilde{\gamma}}{2}}e_{\ell_0+t}(x)p_\delta(x) \lambda^{\tilde\gamma-\zeta}} dy \bigg\|_{L^2}\Big)\;.
	\end{align*}
	Given two models $q,\bar q$, we introduce the distance
	\begin{align*}
		\|\textrm{h};\bar{\textrm{h}}\|_{\tilde{\gamma},\tilde{\eta},\ell_0,\delta,t} &:= \sum_{\zeta \in \cA_{<\tilde\gamma}} \Big(\bigg\| \frac{ \vert \textrm{h}(x)-\bar{\textrm{h}}(x) \vert_\zeta}{t^{\frac{\tilde\eta-\zeta}{2}} e_{\ell_0+t}(x)p_\delta(x)} \bigg\|_{L^2}\\
	&+ \sup_{\lambda \in (0,\sqrt{t/2}]} \bigg\|\int_{y\in B(x,\lambda)} \lambda^{-3} \frac{\vert \textrm{h}(y) -\bar{\textrm{h}}(y) - \Gamma_{y,x} \textrm{h}(x) +\bar\Gamma_{y,x} \bar{\textrm{h}}(x)\vert_\zeta}{t^{\frac{\tilde\eta-\tilde{\gamma}}{2}} e_{\ell_0+t}(x)p_\delta(x) \lambda^{\tilde\gamma-\zeta}} dy \bigg\|_{L^2}\Big)\;.
	\end{align*}
	
	\begin{proposition}\label{Prop:Recons}
	For any $t>0$ and any model $q$, there exists a unique continuous linear map $\cR_t=\cR_t(q)$ from $\ccD^{\tilde{\gamma},\tilde\eta}_{\ell_0,\delta,t}(q)$ into $\cB^{\alpha}_{\ell_0+t,\delta+2\kappa}$ that satisfies
	\begin{equation}\label{Eq:Recons}
		\| \cR_t \textup{h} \|_{\cB^{\alpha}_{\ell_0+t,\delta+2\kappa}} \lesssim t^{\frac{\tilde\eta}{2}}(\|q\|+\|q\|^2) \|\textup{h}\|_{\tilde{\gamma},\tilde{\eta},\ell_0,\delta,t}\;,
	\end{equation}
	as well as
	\begin{equation}\label{Eq:ReconsBd}
		\bigg\| \sup_{\eta \in \cB_r} \frac{\Big\vert \langle \cR_t \textup{h} - \Pi_x \textup{h}(x), \eta^\lambda_x \rangle \Big\vert}{e_{\ell_0+t}(x)p_{\delta+2\kappa}(x)} \bigg\|_{L^2} \lesssim \lambda^{\tilde\gamma} t^{\frac{\tilde\eta-\tilde\gamma}{2}} (\|q\|+\|q\|^2) \|\textup{h}\|_{\tilde{\gamma},\tilde{\eta},\ell_0,\delta,t}\;,
	\end{equation}
	uniformly over all $\lambda \in (0,\sqrt{t/2}]$, all $t>0$ and all models $q$. Furthermore\BLUE{
	\begin{align*}
		&\| \cR_t \textup{h} -\bar\cR_t \bar{\textup{h}}\|_{\cB^{\alpha}_{\ell_0+t,\delta+2\kappa}}\\
		&\lesssim t^{\frac{\tilde\eta}{2}} \Big((\|q\|+\|q\|^2) \|\textup{h};\bar{\textup{h}}\|_{\tilde{\gamma},\tilde{\eta},\ell_0,\delta,t} + \|q-\bar{q}\|(1+\|q\|+\|\bar{q}\|)\|\bar{\textup{h}}\|_{\tilde{\gamma},\tilde{\eta},\ell_0,\delta,t}\Big)\;,
	\end{align*}
	as well as}
	\begin{align*}
		&\bigg\| \sup_{\eta \in \cB_r} \frac{\Big\vert \langle \cR_t \textup{h} - \bar\cR_t \bar{\textup{h}} - \Pi_x \textup{h}(x) + \bar\Pi_x \bar{\textup{h}}(x), \eta^\lambda_x \rangle \Big\vert}{e_{\ell_0+t}(x)p_{\delta+2\kappa}(x)}  \bigg\|_{L^2}\\
	&\lesssim \lambda^{\tilde\gamma} t^{\frac{\tilde\eta-\tilde\gamma}{2}} \Big((\|q\|+\|q\|^2) \|\textup{h};\bar{\textup{h}}\|_{\tilde{\gamma},\tilde{\eta},\ell_0,\delta,t} + \|q-\bar{q}\|(1+\|q\|+\|\bar{q}\|)\|\bar{\textup{h}}\|_{\tilde{\gamma},\tilde{\eta},\ell_0,\delta,t}\Big)\;,
	\end{align*}
	uniformly over all $\lambda \in (0,\sqrt{t/2}]$, and over all $t>0$ and all models $q,\bar q$.
	\end{proposition}
	\begin{proof}
	The reconstruction theorem stated in~\cite[Theorem 2.10]{HL18} deals with space-time modelled distributions but its proof applies verbatim to the case of modelled distributions in space only. It yields the bound
	\begin{align*}&\bigg\| \sup_{\eta \in \cB_r} \Big\vert \langle \cR_t \textup{h} - \Pi_x \textup{h}(x), \eta^\lambda_x \rangle \Big\vert \bigg\|_{L^2(B(x_0,1),dx)}\\
		&\lesssim \lambda^{\tilde\gamma} t^{\frac{\tilde\eta-\tilde\gamma}{2}} (p_\kappa(x_0)+p_\kappa(x_0)^2) e_{\ell_0+t}(x_0)p_{\delta}(x_0) (\|q\|+\|q\|^2) \|\textup{h}\|_{\tilde{\gamma},\tilde{\eta},\ell_0,\delta,t,x_0}\;,
	\end{align*}
	uniformly over all $\lambda \in (0,\sqrt{t/2}]$, all $x_0\in\R^3$. Here  $\|\textup{h}\|_{\tilde{\gamma},\tilde{\eta},\ell_0,\delta,t,x_0}$ is obtained from the expression of $\|\textup{h}\|_{\tilde{\gamma},\tilde{\eta},\ell_0,\delta,t}$ by restricting the $L^2$ norm to the spatial domain $B(x_0,2)$. \BLUE{Since $p_\kappa^2\,p_\delta = p_{2\kappa + \delta}$, }we deduce the bound \eqref{Eq:ReconsBd} of the statement.\\
	Let us prove that $\cR_t$ takes values in $\cB^{\alpha}_{{\ell_0+t},{\delta+2a}}$ and satisfies \eqref{Eq:Recons}. We observe that uniformly over all $\lambda \le \sqrt{t/2}$
	\begin{align*}
		\bigg\| \sup_{\eta \in \cB_r} \frac{\Big\vert \langle\Pi_x \textup{h}(x), \eta^\lambda_x \rangle \Big\vert}{\lambda^{\alpha} e_{\ell_0+t}(x)p_{\delta+2\kappa}(x)} \bigg\|_{L^2}&\lesssim \|q\| \sum_{\zeta} \frac{\lambda^{\zeta} t^{\frac{\tilde\eta-\zeta}{2}}}{\lambda^{\alpha}} \bigg\| \frac{p_\kappa(x)\big\vert \textup{h}(x)\big\vert_\zeta}{ e_{\ell_0+t}(x)p_{\delta+2\kappa}(x)t^{\frac{\tilde\eta-\zeta}{2}}} \bigg\|_{L^2}\;,
	\end{align*}
	Recall that $\zeta \ge \alpha$ hence the bound $\lambda^{\zeta-\alpha}t^{\frac{\tilde\eta-\zeta}{2}} \lesssim t^{\frac{\tilde\eta-\alpha}{2}} \lesssim t^{\frac{\tilde\eta}{2}}$, and we deduce that the last expression is bounded by a term of order $\BLUE{t^{\frac{\tilde\eta}{2}}}\|q\| \|\textup{h}\|_{\tilde{\gamma},\tilde{\eta},\ell_0,\delta,t}$ as required. Together with the reconstruction bound~\eqref{Eq:ReconsBd}, this shows that
	$$ \bigg\| \sup_{\eta \in \cB_r} \frac{\Big\vert \langle \cR_t \textup{h}, \eta^\lambda_x \rangle \Big\vert}{\lambda^{\alpha} e_{\ell_0+t}(x)p_{\delta+2\kappa}(x)} \bigg\|_{L^2} \lesssim t^{\frac{\tilde\eta}{2}}\|\textup{h}\|_{\tilde{\gamma},\tilde{\eta},\ell_0,\delta,t} \BLUE{(\|q\|+\|q\|^2)}\;,$$
	uniformly over all $\lambda \in (0,\sqrt{t/2})$. On the other hand, for $\lambda \in [\sqrt{t/2},1]$, one can write
	$$\eta^\lambda_x(\cdot) = \sum_{y\in (\lambda_0 \bbZ)^3} \eta^\lambda_x(\cdot) \psi((\cdot-y)/\lambda_0)\;,$$
	where $\sum_{y\in \bbZ^3} \psi(\cdot-y)$ is a smooth, compactly supported partition of unity, and $\lambda_0 = \sqrt{t\BLUE{/2}}$. The number of non-zero terms in this sum is of order $(\lambda/\lambda_0)^3$, and each $\eta^\lambda_x(\cdot) \psi((\cdot-y)/\lambda_0)$ can be seen, up to a multiplicative factor uniformly bounded over all parameters at stake, as a function $(\lambda_0/\lambda)^{3} \varphi^{\lambda_0}_y$ where $\varphi\in\cB_r$. By Jensen's inequality, we thus get
	\begin{align*}
		\Big\vert \langle \cR_t \textup{h}, \eta^\lambda_x \rangle \Big\vert^2 &= \Big\vert \sum_{y\in (\lambda_0 \bbZ)^3}\langle \cR_t \textup{h}, \eta^\lambda_x \psi((\cdot-y)/\lambda_0) \rangle \Big\vert^2\\
		&\lesssim \sum_{y\in (\lambda_0 \bbZ)^3:|y-x|\le \lambda}(\lambda_0/\lambda)^{3} \sup_{\varphi\in\cB_r}\Big\vert \langle\cR_t \textup{h}, \varphi_y^{\lambda_0} \rangle \Big\vert^2\;.
	\end{align*}
	Since $\varphi_y$ can be seen as some function $\psi_z$ for any given $z\in B(y,1)$, we obtain
	\begin{align*}
		\bigg\| \sup_{\eta \in \cB_r} \frac{\Big\vert \langle \cR_t \textup{h}, \eta^\lambda_x \rangle \Big\vert}{\lambda^{\alpha} e_{\ell_0+t}(x)p_{\delta+2\kappa}(x)} \bigg\|_{L^2} &\lesssim \bigg(\int_{z\in\R^3} \sup_{\psi\in \cB_r} \Big(\frac{\Big\vert \langle \cR_t \textup{h}, \psi^{\lambda_0}_z \rangle \Big\vert}{\lambda^{\alpha} e_{\ell_0+t}(z)p_{\delta+2\kappa}(z)}\Big)^2 dz\bigg)^{1/2}\;.
	\end{align*}
	At this point, we write \BLUE{$\cR_t \textup{h} = (\cR_t \textup{h}-\Pi_z  \textup{h}(z)) + \Pi_z  \textup{h}(z)$} and argue separately for the two terms. \BLUE{Applying the reconstruction bound \eqref{Eq:ReconsBd} to the first term, we obtain a bound of} order $t^{\frac{\tilde\eta}{2}} \|\textup{h}\|_{\tilde{\gamma},\tilde{\eta},\ell_0,\delta,t} \lambda^{-\alpha}$ while \BLUE{the second term can be bounded as follows}
	\begin{align*}
		 \bigg(\int_{z\in\R^3} \sup_{\psi\in \cB_r} \Big(\frac{\Big\vert \langle \Pi_z \textup{h}(z), \psi^{\lambda_0}_z \rangle \Big\vert}{\lambda^{\alpha} e_{\ell_0+t}(z)p_{\delta+2\kappa}(z)}\Big)^2 dz\bigg)^{1/2} &\lesssim \|\textup{h}\|_{\tilde{\gamma},\tilde{\eta},\ell_0,\delta,t} \sum_\zeta \Big(\frac{t^{\frac{\tilde{\eta}-\zeta}{2}} \lambda_0^\zeta}{\lambda^\alpha}\Big)\\
		 &\lesssim  \|\textup{h}\|_{\tilde{\gamma},\tilde{\eta},\ell_0,\delta,t} \,t^{\frac{\tilde{\eta}}{2}} \lambda^{-\alpha}\\
		 &\lesssim \|\textup{h}\|_{\tilde{\gamma},\tilde{\eta},\ell_0,\delta,t} \,t^{\frac{\tilde{\eta}}{2}}\;,
	\end{align*}
	uniformly over all $\lambda \in [\sqrt{t/2},1]$\BLUE{, thus concluding the proof that $\cR_t$ takes values in $\cB^{\alpha}_{{\ell_0+t},{\delta+2a}}$ and satisfies \eqref{Eq:Recons}. The case of two models follow from similar arguments.}
	\end{proof}
	\BLUE{We now proceed with the proof of Proposition \ref{Prop:ContinuityInTime}.}
	\begin{proof}[Proof of Proposition \ref{Prop:ContinuityInTime}]
	Recall that $\textrm{u}$ is the unique fixed point of~\eqref{Eq:FixedPt3d} that was constructed in~\cite[Thm 5.2]{HL18}. Let us write $\textrm{f} := \textrm{u}\Xi$ and recall from~\cite[Section 4.1]{HL18} that $\textrm{f}\in\cD^{\gamma',\eta',2}_{T,w}(q)$ where $\gamma':=\gamma+\alpha$, $\eta':=\eta+\alpha<\alpha$. The definition of this space was given previously, let us observe that for all $\zeta\in\cA_{<\gamma'}$ the following quantities are finite
	\begin{align}
		&\sup_{t\in (0,T]} \bigg\| \frac{\vert \textrm{f}(t,x) \vert_\zeta}{t^{\frac{\eta'-\zeta}{2}} w^{(1)}_t(x,\zeta)} \bigg\|_{L^2} \\
		&\sup_{\lambda \in (0,2]}\sup_{t\in (2\lambda^2,T]} \bigg\|\int_{y\in B(x,\lambda)} \lambda^{-3} \frac{\vert \textrm{f}(t,y) - \Gamma_{y,x} \textrm{f}(t,x) \vert_\zeta}{t^{\frac{\eta'-\gamma'}{2}} w^{(2)}_t(x,\zeta) \lambda^{\gamma'-\zeta}} dy \bigg\|_{L^2} \label{Eq:fSpaceReg}\\
		&\sup_{t\in (0,T]}\sup_{s \in (t/2,t)} \bigg\|\frac{\vert \textrm{f}(t,x) - \textrm{f}(s,x) \vert_\zeta}{t^{\frac{\eta'-\gamma'}{2}} w^{(1)}_t(x,\zeta) \vert t-s\vert^{\frac{\gamma'-\zeta}{2}}} \bigg\|_{L^2} \label{Eq:fTimeReg}
	\end{align}
	The precise expressions of the weights are unimportant for this proof, we only need that there exists some constants $c,C>0$ such that for all $t,x$ and all $\zeta$, $i\in\{1,2\}$
	$$ (1+p_\kappa(x))w^{(i)}_t(x,\zeta) \le C e_{\ell_0+t}(x) p_c(x)\;,$$
	and that $c$ can be taken as small as desired by diminishing $\kappa$.\\
	\BLUE{We set $\tilde{\gamma} = \gamma'-\kappa/2 > 0$ and $\tilde{\eta} = \eta'-\kappa/2 > -2$. The bounds recalled above show that $\textrm{f}(t,\cdot) \in \ccD^{\tilde{\gamma},\tilde\eta}_{\ell_0,c,t}(q)$ so that Proposition \ref{Prop:Recons} yields
	\begin{equation*}
		\| \cR_t(\textrm{f}(t,\cdot)) \|_{\cB^{\alpha}_{\ell_0 + t,c+2\kappa}} \lesssim t^{\frac{\tilde\eta}{2}}\;,
	\end{equation*}
	uniformly over all $t\in (0,T)$.}\\
	We turn to the continuity in time. We consider $s \in (t/2,t)$. \BLUE{It is straightforward to check that $\textrm{f}(s,\cdot) \in \ccD^{\tilde{\gamma},\tilde\eta}_{\ell_0,\delta,t}(q)$ so that}
	$$\cR_t(\textrm{f}(t,\cdot)) - \cR_s(\textrm{f}(s,\cdot)) = \cR_t(\textrm{f}(t,\cdot)) - \cR_t(\textrm{f}(s,\cdot)) = \cR_t(\textrm{f}(t,\cdot)-\textrm{f}(s,\cdot))\;.$$
	\BLUE{We would like to establish that
	$$ 	\| \cR_t(\textrm{f}(t,\cdot)-\textrm{f}(s,\cdot)) \|_{\cB^{\alpha}_{\ell_0+t,c+2\kappa}} \lesssim |t-s|^{\frac{\kappa}{4}} t^{\frac{\tilde\eta}{2}}\;,$$
	uniformly over all $t\in (0,T]$ and all $s\in (t/2,t)$. Given~\eqref{Eq:Recons}, it suffices to show that
	$$ \| \textrm{f}(t,\cdot)-\textrm{f}(s,\cdot) \|_{\tilde{\gamma},\tilde\eta,\ell_0,c,t} \lesssim |t-s|^{\frac{\kappa}{4}}\;.$$
}
	From \eqref{Eq:fTimeReg}, we deduce that
	\begin{align*}
		\bigg\| \frac{\vert \textrm{f}(t,x)-\textrm{f}(s,x) \vert_\zeta}{t^{\frac{\tilde\eta-\zeta}{2}}e_{\ell_0+t}(x)p_c(x)} \bigg\|_{L^2} &\lesssim \frac{t^{\frac{\eta'-\gamma'}{2}}}{t^{\frac{\tilde\eta-\zeta}{2}}}|t-s|^{\frac{\gamma'-\zeta}{2}} =t^{\frac{\zeta-\tilde\gamma}{2}}|t-s|^{\frac{\gamma'-\tilde\gamma}{2}}|t-s|^{\frac{\tilde\gamma-\zeta}{2}}\\
		&\lesssim |t-s|^{\frac{\kappa}{4}}\;,
	\end{align*}
	\BLUE{uniformly over all $\zeta\in \cA_{<\tilde\gamma}$, all $s\in (t/2,t)$ and all $t\in (0,T]$.}\\
	Regarding the regularity in space, we distinguish two cases. First if $\lambda \le \sqrt{t-s}$ then we bound separately the contributions coming from $\textrm{f}(t,\cdot)$ and $\textrm{f}(s,\cdot)$
	\begin{align*}
		&\bigg\|\int_{y\in B(x,\lambda)} \lambda^{-3} \frac{\vert \textrm{f}(t,y) -\textrm{f}(s,y) - \Gamma_{y,x} \textrm{f}(t,x) + \Gamma_{y,x} \textrm{f}(s,x) \vert_\zeta}{t^{\frac{\tilde\eta-\tilde\gamma}{2}}e_{\ell_0+t}(x)p_c(x) \lambda^{\tilde\gamma-\zeta}} dy \bigg\|_{L^2(\R^3)}\\
		&\lesssim \bigg\|\int_{y\in B(x,\lambda)} \lambda^{-3} \frac{\vert \textrm{f}(t,y) - \Gamma_{y,x} \textrm{f}(t,x)\vert_\zeta}{t^{\frac{\tilde\eta-\tilde\gamma}{2}}e_{\ell_0+t}(x)p_c(x) \lambda^{\tilde\gamma-\zeta}} dy \bigg\|_{L^2(\R^3)}\\
		&+\bigg\|\int_{y\in B(x,\lambda)} \lambda^{-3} \frac{\vert \textrm{f}(s,y) - \Gamma_{y,x} \textrm{f}(s,x)\vert_\zeta}{t^{\frac{\tilde\eta-\tilde\gamma}{2}}e_{\ell_0+t}(x)p_c(x) \lambda^{\tilde\gamma-\zeta}} dy \bigg\|_{L^2(\R^3)}
	\end{align*}
	From \eqref{Eq:fSpaceReg} we get a bound of order
	$$\lambda^{\gamma'-\tilde{\gamma}} \frac{t^{\frac{\eta'-\gamma'}{2}}}{t^{\frac{\tilde\eta-\tilde\gamma}{2}}} \le (t-s)^{\frac{\kappa}{4}}  \;,$$
	uniformly over all $\lambda \le \sqrt{t-s} \wedge \sqrt{t/2}$, all $s\in (t/2,t)$ and all $t\in (0,T)$. Second if $\sqrt{t-s}\le \lambda \le \sqrt{t/2}$ then we write
		\begin{align*}
		&\bigg\|\int_{y\in B(x,\lambda)} \lambda^{-3} \frac{\vert \textrm{f}(t,y) -\textrm{f}(s,y) - \Gamma_{y,x} \textrm{f}(t,x) + \Gamma_{y,x} \textrm{f}(s,x) \vert_\zeta}{t^{\frac{\tilde\eta-\tilde\gamma}{2}}e_{\ell_0+t}(x)p_c(x) \lambda^{\tilde\gamma-\zeta}} dy \bigg\|_{L^2(\R^3)}\\
		&\leq \bigg\|\int_{y\in B(x,\lambda)} \lambda^{-3} \frac{\vert \textrm{f}(t,y) -\textrm{f}(s,y)\vert_\zeta}{t^{\frac{\tilde\eta-\tilde\gamma}{2}}e_{\ell_0+t}(x)p_c(x) \lambda^{\tilde\gamma-\zeta}} dy \bigg\|_{L^2(\R^3)}\\
		&+\bigg\|\int_{y\in B(x,\lambda)} \lambda^{-3} \frac{\vert \Gamma_{y,x} \textrm{f}(t,x) - \Gamma_{y,x} \textrm{f}(s,x)\vert_\zeta}{t^{\frac{\tilde\eta-\tilde\gamma}{2}}e_{\ell_0+t}(x)p_c(x) \lambda^{\tilde\gamma-\zeta}} dy \bigg\|_{L^2(\R^3)}
	\end{align*}
	The first term can be bounded using \eqref{Eq:fTimeReg} and yields a bound of order
	$$ \frac1{\lambda^{\tilde\gamma -\zeta}} |t-s|^{\frac{\gamma'-\zeta}{2}} \frac{t^{\frac{\eta'-\gamma'}{2}}}{t^{\frac{\tilde\eta-\tilde\gamma}{2}}} \leq |t-s|^{\frac{\gamma'-\tilde\gamma}{2}} = |t-s|^{\frac{\kappa}{4}} \;,$$
	uniformly over all $\sqrt{t-s} \le \lambda \le \sqrt{t/2}$, all $s\in (t/2,t)$ and all $t\in (0,T)$. Regarding the second term, we use the bound
	$$\vert \Gamma_{y,x} \tau\vert_\zeta \lesssim \sum_{\beta \ge \zeta} p_\kappa(x) \vert y-x \vert^{\beta-\zeta} \vert \tau|_\beta \le \sum_{\beta \ge \zeta} p_\kappa(x) \lambda^{\beta-\zeta}\vert \tau|_\beta\;,$$
	together with \eqref{Eq:fTimeReg} to bound the second term by the sum over all $\beta \ge \zeta$ of
	\begin{align*}
		\bigg\|\int_{y\in B(x,\lambda)} \lambda^{-3} \frac{p_\kappa(x) \vert \textrm{f}(t,x) - \textrm{f}(s,x)\vert_\beta}{t^{\frac{\tilde\eta-\tilde\gamma}{2}}e_{\ell_0+t}(x)p_c(x) \lambda^{\tilde\gamma-\beta}} dy \bigg\|_{L^2}&\lesssim \frac{(t-s)^{\frac{\gamma'-\beta}{2}}}{\lambda^{\tilde\gamma-\beta}} \frac{t^{\frac{\eta'-\gamma'}{2}}}{t^{\frac{\tilde\eta-\tilde\gamma}{2}}}\\
		&\le (t-s)^{\frac{\kappa}{4}}\;,
	\end{align*}
	uniformly over all $\sqrt{t-s} \le \lambda \le \sqrt{t/2}$, all $s\in (t/2,t)$ and all $t\in (0,T)$.\\
	\BLUE{We have thus proven that $g:t\mapsto \cR_t (f(t,\cdot))$ lies in $\ccE^{\alpha,\eta+\alpha-\kappa/2}_{\ell_0,\delta, T}$. We need to argue that it coincides with $\cR(f)$. Note that, if  $\eta:\R\times \R^3\to\R$ is a space-time test function, then
		$$ \langle g, \eta\rangle = \int_{s\in\R} \langle \cR_s(f(s,\cdot)), \eta(s,\cdot)\rangle ds\;.$$
	Then the reconstruction bound~\eqref{Eq:ReconsBd} ensures that
	\begin{align*}
		\ &\bigg\| \sup_{\eta \in \cB_r(\R^4)} \Big\vert \int_s \Big\langle \cR_s \textup{f}(s,\cdot) - \Pi_x \textup{f}(s,x), \lambda^{-5} \eta\Big(\frac{s-t}{\lambda^2},\frac{\cdot-x}{\lambda}\Big) \Big\rangle ds \Big\vert \bigg\|_{L^2(B(x_0,1),dx)}\\
		&\lesssim \int_{s\in [t-\lambda^2,t+\lambda^2]} \lambda^{\tilde{\gamma}} s^{\frac{\tilde\eta-\tilde\gamma}{2}} e_{\ell_0+s}(x_0)p_{c+2\kappa}(x_0) \lambda^{-2} ds\\
		&\lesssim \lambda^{\tilde{\gamma}} t^{\frac{\tilde\eta-\tilde\gamma}{2}} e_{\ell_0+t+\lambda^2}(x_0)p_{c+2\kappa}(x_0)\;,
	\end{align*}
	uniformly over all $t\in (3\lambda^2,T-\lambda^2]$, all $\lambda\in (0,1]$, all $x_0\in\R^3$. The uniqueness associated with the reconstruction bound of~\cite[Eq (3.12) of Theorem 3.10]{HL18} ensures that $g=\cR(\textrm{f})$.\\
	Finally, the continuity with respect to $(q,f)$ can be deduced from the continuity of $(q,f) \mapsto \textrm{u}$ and the continuity bounds stated in Proposition \ref{Prop:Recons}.}
	\end{proof}

\subsection{The self-adjoint operator}

\begin{definition}\label{def:symmetric-semigroup_3d}
	Fix $q \in \ccM$. For $t>0$, define the domain $\cD_t := L^2_{e_{-t}} \subset L^2$ and the operator
	\[P_t : \bigg\{ \begin{array}{lll}
		\cD_t &\to &L^2\\
		f & \mapsto &u^{q, f}(t)
	\end{array}\]
	Define also $P_0 = I$ on $L^2$.
\end{definition}

\begin{proposition}
	Fix $q\in \ccM$. The collection $(P_t)_{t \in [0, T]}$ is a semigroup of operators on $\mathfrak{h} = L^2(\R^3,dx)$ and is strongly continuous with respect to $t$. In particular, it satisfies the \emph{(Non-increasing domains with dense union)}, \emph{(Semigroup)}, \emph{(Symmetry)} and \emph{(Weak continuity)} properties of Theorem \ref{Th:Klein-Landau}.
\end{proposition}
\begin{proof}
	{\emph{Non-increasing domains with dense union}:} It is immediate that $\cD_t \subset \cD_s$ whenever $s\le t$. Furthermore since $C_c \subset \cD_t$ the density property holds. {\emph{Semigroup}:} This is a consequence of items 1 and 3 of Theorem \ref{Th:PAM3d}. {\emph{Weak-continuity}:} This is a consequence of item 2 of Theorem \ref{Th:PAM3d} that even shows strong continuity in time. {\emph{Symmetry}:} This is a consequence of item 4 of Theorem \ref{Th:PAM3d} and of the continuity stated in item 1 of Theorem \ref{Th:PAM3d}.
\end{proof}

\begin{definition}\label{def:L_3d}
	For $q \in {\ccM}$, let $H(q)$ be the unique self-adjoint operator given by Theorem \ref{Th:Klein-Landau} associated to the symmetric semigroup $(P_t)_{t \in [0, T]}$ of Definition \ref{def:symmetric-semigroup}. In particular, $P_t$ is the restriction of $e^{-tH(q)}$ to $\cD_t$ for all $t \in [0, T]$.
\end{definition}

\begin{remark}
Since $\cD_t$ is dense in $L^2$, the proof of~\cite[Lemma 6]{KL81} shows that $H(q)$ is essentially self-adjoint over $\hat{\cD}(q) := \bigcup_{0 < s < t} P_s \cD_t$.
\end{remark}

We now establish a few properties satisfied by $H(q)$. Given any $f\in L^2$, denote by $\mu_f(q)$ the spectral measure associated to the self-adjoint operator $H(q)$ and $f$, that is, the unique finite measure on $\R$ with Stieltjes transform
$$ \int_\R (\lambda-z)^{-1} \mu_f(q)(d\lambda) = \langle f, (H(q)-z)^{-1} f\rangle\;,\quad \forall z\in \C\backslash \R\;.$$

\begin{proposition}\label{prop:continuity_H_3d}
	\begin{enumerate}
		\item For any $f\in L^2$, the map $q \mapsto \mu_f(q)$ is continuous from ${\ccM}$ into the space of finite non-negative measures endowed with the topology of weak convergence. As a consequence $q\mapsto H(q)$ is continuous in the strong resolvent sense.
		\item When $q=Q_\eps(\xi)$, $H(q)$ coincides with the essentially self-adjoint operator $-\Delta + \xi_\eps + C_\eps$.
	\end{enumerate}
\end{proposition}
\begin{proof}
	Fix $f\in C_c$. Then $f\in \cD_t(q)$ for any $q\in {\ccM}$ and any $t\ge 0$. In addition, we have
	\[\int_\R e^{-s \lambda} \mu_f(q)(\dd \lambda) = \crochet{f, e^{-sH(q)} f} = \crochet{f, u^{q,f}(s)}, \quad \forall s \in [0, t] \;.\]
	Suppose $(q_n)$ is a sequence in ${\ccM}$ converging to $q$. Since $f$ has compact support, we deduce from item 1 of Theorem \ref{Th:PAM3d} that
	$$ \crochet{f, u^{q_n, f}(s)} \to \crochet{f, u^{q, f}(s)}\;,\quad n\to\infty\;.$$
	Therefore the Laplace transform of the measure $\mu_f(q_n)$ converges to the Laplace transform of $\mu_f(q)$. By Lemma \ref{Lemma:Laplace}, this implies weak convergence of the measures. The rest of the proof is identical to the proof of Proposition \ref{prop:continuity_H}.
\end{proof}

We have all the ingredients at hand to define the Anderson Hamiltonian with white noise potential on $\R^3$. For any $\xi\in\Omega$, recall that $Q_\eps(\xi), Q(\xi) \in {\ccM}$ and set
\[\cH_\eps(\xi) := H(Q_\eps(\xi))\;,\quad \cH(\xi) := H(Q(\xi))\;,\]
where $H$ is the deterministic map introduced in Definition \ref{def:L_3d}.

\begin{proof}[Proof of Theorem \ref{Th:Construction} in dimension $3$]
	By Proposition \ref{prop:continuity_H_3d}, for any $\xi\in\Omega$ we have $\cH_\eps(\xi) = -\Delta + \xi_\eps + C_\eps$. By Lemma \ref{lem:Omega03d}, $Q$ is the limit in probability of $Q_\eps$ so that Proposition \ref{prop:continuity_H_3d} entails that $\cH_\eps$ converges in the strong resolvent sense to $\cH$ in probability. Finally, Definition \ref{def:L_3d} ensures that, for any $t\ge 0$, the domain of the operator $e^{-t\cH}$ contains the set $\cD_t$ so in particular all functions $f\in L^2(\R^d,dx)$ with compact support, and $e^{-t\cH} f$ coincides with the solution of \eqref{Eq:PAM} at time $t$, starting from $f$ at time $0$.
\end{proof}

\subsection{Inclusion of the supports}\label{subsec:supports3d}
 Let $\supp(Q_\eps)$ and $\supp(Q)$ be the topological supports of the laws of the r.v.~$Q_\eps$ and $Q$. The following result is due to Hairer and Sch\"onbauer~\cite{HS21}.
 
 \begin{theorem}\label{Th:Support3d}
 	For any given $\eps \in (0,1)$, it holds $\supp(Q_\eps) \subset \supp(Q)$.
 \end{theorem}

\BLUE{Let us provide some general comments. In~\cite{HS21}, Hairer and Sch\"onbauer present a general framework that allows to determine the support of the solutions of singular SPDEs. Previously, Chouk and Friz~\cite{CF16} identified the support of the solution of the generalised (PAM) in dimension $2$. In both references, a crucial step consists in identifying the support of the enhanced noise, also called model in the context of regularity structures. This support involves some constants, that can be interpreted as renormalisation constants: in the context of Theorem \ref{Th:Support2d}, it turns out that the support contains $(0,c)$ for any constant $c$, and a major difficulty of the proof consisted in showing that $(0,c)$ indeed belongs to the support for any given $c\in\R$. In the specific case of the generalised (PAM) in dimension $2$, Chouk and Friz~\cite{CF16} presented a very explicit proof: they managed to produce any renormalisation constant by cooking up some sort of ``deterministic pure area rough path''. Surprinsingly, their proof exploited the Fourier analysis of the paracontrolled calculus extensively and it is unclear how to adapt their proof in the context of regularity structures. Hairer and Sch\"onbauer~\cite{HS21} proposed a different approach for ``producing'' the renormalisation constants: instead of cooking up a ``deterministic pure area rough path'', they shifted the model by a random, stationary, smooth noise (built from the original white noise by a local procedure). In the case of dimension $2$, we implemented this approach in Lemma \ref{Lemma:Resonance}. The case of dimension $3$ is much more involved, and falls into the framework of~\cite{HS21}. Let us mention that Hairer and Sch\"onbauer work on a torus. However their arguments are local in nature so that they apply also to the full space setting.}
\begin{proof}
	We aim at proving that
	\begin{equation}\label{Eq:supp3d}
	\supp(Q) = \overline{\{\ccR^{\bar c} q_{\mbox{\tiny can}}(h) : \bar{c} \in\BLUE{\R^2},\quad h\in  C^\infty_{p_{b}}\}}^{\ccM}\;.
	\end{equation}
	Recall that for all $\xi\in\Omega$, $Q_\eps(\xi) = \ccR^{\bar c_\eps} q_{\mbox{\tiny can}}(\xi_\eps)$ with $\bar c_\eps = (c_\eps,\BLUE{c^{(1)}_\eps})$ and $\xi_\eps \in C^\infty_{p_b}$. As a consequence $\supp(Q_\eps)$ belongs to the set on the r.h.s.~of \eqref{Eq:supp3d}. Consequently, the statement of the theorem is proven once \eqref{Eq:supp3d} is established. To prove this identity, we rely on three ingredients:\begin{enumerate}
		\item There exists a shift operator $(h,q) \mapsto T_h q$ which is continuous from $C^\infty_{p_{b}}\times \ccM$ into $\ccM$ and which satisfies
		$$ \ccR^{\bar c} q_{\mbox{\tiny can}}(f+h) = T_h \ccR^{\bar c} q_{\mbox{\tiny can}}(f) = \ccR^{\bar c} T_h q_{\mbox{\tiny can}}(f)\;,$$
		for all $\bar c\in\R^3$ and all $f,h \in C^\infty_{p_b}$.
		\item If $q\in \supp(Q)$ then $T_h q \in \supp(Q)$ for any $h\in C^\infty_{p_{b}}$.
		\item For any $\bar c\in\BLUE{\R^2}$, $\ccR^{\bar c}q(0) \in \supp(Q)$.
	\end{enumerate}
	With these three ingredients at hand, we deduce that for any $\bar c\in\R^3$ and any $h\in  C^\infty_{p_{b}}$
	$$ \ccR^{\bar c} q_{\mbox{\tiny can}}(h) = T_h \ccR^{\bar c} q_{\mbox{\tiny can}}(0) \in \supp(Q)\;,$$
	and since $\supp(Q)$ is a closed set, \BLUE{we deduce that $\supp(Q)$ contains the set on the r.h.s.~of \eqref{Eq:supp3d}. On the other hand, since $Q_\eps$ converges in law to $Q$ and since we already know that $\supp(Q_\eps)$ belongs to the set on the r.h.s.~of \eqref{Eq:supp3d}, we deduce that $\supp(Q)$ also belongs to the set on the r.h.s.~of \eqref{Eq:supp3d} and this completes the proof.}\\
	Let us provide arguments for the three ingredients above. Item 1.~is an extension to a weighted setting of~\cite[Theorem 2.4]{HS21}: the algebraic part works exactly the same while the required analytic bounds in infinite volume are satisfied since we chose the parameter $b$ \BLUE{smaller than $\kappa/4$}. Let us prove Item 2. From the continuity of the shift operator and the closedness of the support, it is sufficient to deal with $h\in C^\infty_{p_{b}} \cap L^2$. Then, the Cameron-Martin theorem ensures that the laws of $\xi$ and $\xi+h$ are equivalent. Furthermore \BLUE{for all $\xi \in \Omega$ and all $\eps >0$ we have
		$$ Q_\eps(\xi+h) = T_h Q_\eps(\xi)\;.$$
	In addition, the convergence in probability of $Q_\eps$ to $Q$ and the continuity of the shift operator shows that $\lim_\eps T_h Q_\eps(\xi) = T_h Q(\xi)$ in probability. The equivalence of the laws of $\xi$ and $\xi+h$ and the convergence in probability of $Q_\eps$ to $Q$ shows that $\lim_\eps Q_\eps(\xi+h)=Q(\xi+h)$ in probability. Therefore, for $\P$-almost all $\xi$
	$$ Q(\xi+h) = T_h Q(\xi)\;.$$}
	As a consequence the laws of $T_h Q(\xi)$ and $Q(\xi)$ are equivalent and $\supp(Q) = \supp(T_h Q)$. But a general result, see~\cite[Lemma 3.13]{HS21}, shows that $\supp(T_h Q)$ is the closure of $T_h(\supp(Q))$. Consequently if $q\in \supp \BLUE{(Q)}$, then $T_h q\in \supp(T_h Q)=\supp(Q)$.\\
	Item 3. can be deduced from a general result of Hairer and Sch\"onbauer~\cite{HS21}. To apply this result, we need to check that the Assumptions 2-6 from~\cite{HS21} are satisfied by the regularity structure for (PAM) in dimension $3$. Assumption 2, which is necessary for the BPHZ theorem \BLUE{ as it avoids variance blow-up}, is satisfied. Assumption 3 is trivially satisfied since we do not consider product of noises nor derivative of noise. Assumption 4, which requires the integration kernel to be homogeneous, holds. Regarding Assumptions 5 and 6, we have $\cV=\cV_0 = \{\Xi \cI(X_i \Xi)\}$ and the symmetries of the integration kernel imply the desired assumption\footnote{The symmetry group introduced in~\cite[Subsection 2.5]{HS21} and required in the Assumptions 5 and 6 is taken to be the finite group of transformations of $\R^3$ generated by $x\mapsto x$ and $x\mapsto -x$: note that the Green function of the Laplacian is invariant under these tranformations.}.\\
	In the context of~\cite[Def. 3.3]{HS21}, the annihilator $\ccH$ is the whole renormalisation group $\cG_-$. The desired result is then stated as~\cite[Proposition 3.8]{HS21}. The strategy of proof, presented in~\cite[Prop 3.20 and 3.21]{HS21}, consists in producing a random shift $\zeta_\delta$ which is such that $T_{\zeta_\delta} Q(\xi) \to \ccR^{\bar c}q(0)$ in probability as $\delta\downarrow 0$. The shift $\zeta_\delta$ happens to be a stationnary field living in a finite inhomogeneous Wiener chaos: it is associated, through It\^o-Wiener isometry, to compactly supported kernels. The convergence proven in~\cite{HS21} thus extends to the weighted setting considered in the present article using the same argument as in \eqref{Eq:xieps2d} as we did for the convergence of the renormalised model.
\end{proof}

\section{Identification of the spectrum}\label{sec:spectrum}

In this section, we identify the spectrum of the Anderson Hamiltonian with white noise potential. Our arguments apply simultaneously in dimensions $2$ and $3$.

For any $x\in\R^d$, let $\cT_x$ denote the translation operator, that is, the operator
$$ \cT_x f(y) := f(y + x)\;,\quad y\in\R^d\;.$$
Each $\cT_x$ is unitary (with adjoint $\cT_x^*f = \cT_x^{-1}f = f(\cdot - x)$) from $L^2(\R^d)$ into itself.

\begin{lemma}\label{lem:shiftshift}
	For all $\xi \in\Omega_0$, all $x\in \R^d$ and all $\eps\in (0,1)$
	$$ \cH_\eps(\theta_x \xi) = \cT_x^* \cH_\eps(\xi) \cT_x\;,\quad \cH(\theta_x \xi) = \cT_x^* \cH(\xi) \cT_x\;.$$
\end{lemma}
\begin{proof}
	For all $\xi\in \Omega$, we have seen that in both dimensions $2$ and $3$ the operator $\cH_\eps(\theta_x \xi)$ coincides with $-\Delta + \xi_\eps(\cdot-x) + C_\eps$ so that
	\[\cH_\eps(\theta_x\xi) =  \cT_x^* \cH_\eps(\xi) \cT_x.\]	
	We now restrict to $\xi \in \Omega_0$. In both dimensions, for all $x\in\R^d$
	\[\cH(\theta_x \xi) = \lim_k \cH_{\eps_k}(\theta_x\xi) = \lim_{\eps_k} \cT_x^* \cH_{\eps_k}(\xi) \cT_x = \cT_x^* \cH(\xi) \cT_x\]
	where the limits are in the strong resolvent sense and where we have used the fact that $\cT_x$ is unitary on $L^2$.
\end{proof}

We are now in the framework of ergodic random operators. More precisely, it is well-known that $(\theta_x)_{x \in \R^d}$ is an ergodic family of measure-preserving transformations on the probability space $(\Omega,\P)$, see for instance~\cite[Prop. B.1]{Mat21}. For any $z\in\C\backslash\R$, the maps $\xi \mapsto (\cH_\eps(\xi)-z)^{-1}$ and $\xi\mapsto (\cH(\xi)-z)^{-1}$ are the compositions of the measurable maps $\xi \mapsto Q_\eps(\xi)$ and $\xi\mapsto Q(\xi)$ with the continuous (see Propositions \ref{prop:continuity_H} and \ref{prop:continuity_H_3d}) map $q\mapsto (H(q)-z)^{-1}$: therefore, they are measurable from $(\Omega,\P)$ into the set of bounded operators on $L^2(\R^d,dx)$, and we deduce that the random operators $\cH_\eps$ and $\cH$ are measurable in the sense of~\cite[Definition V.1.3]{CL90}. These properties, combined with the result of the last lemma, imply that the spectra of $\cH_\eps$ and $\cH$ are almost surely deterministic sets, see~\cite[Proposition V.2.2]{CL90}: namely, given $\eps \in (0,1)$, there exist $\Sigma_\eps,\Sigma\subset \R$ such that for $\P$-almost all $\xi$, the spectrum of $\cH_\eps(\xi)$ is $\Sigma_\eps$ and the spectrum of $\cH(\xi)$ is $\Sigma$. (A priori the sets $\Sigma_\eps, \Sigma$ depend on the dimension at stake but, as we will see, they don't).

Our last ingredient is the identification of the almost sure spectrum of $\cH_\eps$: it follows from standard techniques.

\begin{proposition}\label{prop:Sigma_eps}
	For any given $\eps\in (0,1)$ and in dimensions $2$ and $3$, $\Sigma_\eps = \R$.
\end{proposition}
\begin{proof}
	Fix $r \in \R$. The strategy is to construct a Weyl sequence of $\cH_\eps$ for the value $r$, i.e. a sequence $(f_n)$ in $\cD(\cH_\eps)$ such that $\norm{f_n}_{L^2} = 1$ and $\norm{(\cH_\eps - r) f_n}_{L^2} \to 0$. The existence of such a sequence implies that $r\in\Sigma_\eps$ see~\cite[Lemma 2.17]{Tes14}. To do so, we argue by ergodicity that there exist arbitrary large regions of space where $\xi_\eps$ is ``flat'' so that we can use a Weyl sequence of the Laplacian for a well-chosen energy.
	
	For $n \ge 1$, define the event
	\[E_n :=  \left\{ \xi \in \Omega \bigg|  \exists z \in \R^d, \forall y\in \R^d: |y - z| \leq n \Rightarrow |\xi_\eps(y) + C_\eps- r | \leq \frac1{n}\right\}\;,\]
	as well as $E = \bigcap_{n \ge 1} E_n$. We claim that $\P(E) = 1$. First observe that the event $E_n$ is invariant with respect to the shift operators $(\theta_x)_{x \in \R^d}$. The ergodicity of $\xi_\eps$ thus implies that each event $E_n$ has probability either $0$ or $1$. However, the explicit gaussian structure of $\xi_\eps$ allows to show that $\P(E_n) \ge \P(\sup_{|y| \le n} |\xi_\eps(y)+C_\eps - r| \leq \frac1{n}) > 0$. As a consequence $\P(E_n) = 1$ and therefore $\P(E) = \P(\bigcap_n E_n) = 1$.
	
	Fix $\xi \in E$. For any $n\ge 1$, we can find $z_n(\xi) \in \R^d$ such that $\sup_{y:  |y - z_n(\xi)| \leq n}|\xi_\eps(y) - r + C_\eps| \leq 1/n$. Let $\chi \in C^\infty_c$ be such that $0\leq \chi \leq 1$, $\chi \equiv 1$ on $B(0, 1/2)$ and $\chi$ is supported in $B(0, 1)$. \BLUE{Let us define the function
	\[f_n(x) := \frac{\chi\big((x - z_n(\xi))/n\big)}{n^{d/2} \| \chi \|_{L^2}} , \quad x \in \R^d,\]
	which is normalized in $L^2$.} Note that $\norm{\nabla f_n}_{L^2} \leq \frac{C}{n}$ and $\norm{\Delta f_n}_{L^2} \leq \frac{C}{n^2}$ for some constant $C > 0$ independent of $n$. It is clear that $f_n$ is in the domain of $\cH_\eps(\xi) = -\Delta + \xi_\eps + C_\eps$. In addition
	\[\norm{(\cH_\eps - r) f_n}_{L^2} \leq \norm{\Delta f_n}_{L^2}  + \norm{(\xi_\eps + C_\eps - r) f_n}_{L^2} \leq \frac{C}{n^2} + \frac{1}{n} \to 0\;.\]
	Therefore $(f_n)$ is a Weyl sequence for $\cH_\eps(\xi)$ at energy $r$. Since $E$ has full measure, this proves that $r \in \Sigma_\eps$ and thus $\Sigma_\eps = \R$.
\end{proof}
We now identify the almost sure spectrum of $\cH$.
\begin{proof}[Proof of Theorem \ref{Th:Spectrum}]
	Fix $\eps \in (0,1)$. There exists a measurable set $\Omega_1\subset \Omega_0$ of full $\P$-measure on which the spectrum of $\cH_\eps(\xi)$ is $\Sigma_\eps$ and the spectrum of $\cH(\xi)$ is $\Sigma$. Let $q$ be an element in the intersection of $\supp(Q_\eps)$ and $Q_\eps(\Omega_1)$. There exists a sequence $q_n$ that lies in $Q(\Omega_1)$ such that $q_n \to q$ in the space $\ccM$ as $n\to\infty$. Indeed, if no such sequence exists then there is a ball of radius $\delta > 0$ centered at $q$ which is completely included into $\ccM\backslash Q(\Omega_1)$. This ball would have zero-measure under the law of $Q$ and therefore $q$ would not lie in $\supp(Q)$: this would raise a contradiction since $q$ lies in $\supp(Q_\eps)$ which is included into $\supp(Q)$ by Theorem \ref{Th:Support2d} in dimension $2$ and Theorem \ref{Th:Support3d} in dimension $3$.
	
	Necessarily there exist $\xi,\xi_n\in\Omega_1$ such that $q = Q_\eps(\xi)$ and $q_n = Q(\xi_n)$. Fix $f\in L^2(\R^d)$ with unit $L^2$-norm and recall that $\mu_f(q)$, $\mu_f(q_n)$ are the spectral measures associated to $H(q)=\cH_\eps(\xi)$ and $H(q_n)=\cH(\xi_n)$. Note they are probability measures on $\R$. Since $\xi\in\Omega_1$, the spectrum of $\cH_\eps(\xi)$ is $\Sigma_\eps$. Similarly, the spectrum of $\cH(\xi_n)$ is $\Sigma$ for every $n\ge 1$. By the continuity of the deterministic map $H$ stated in Propositions \ref{prop:continuity_H} and \ref{prop:continuity_H_3d}, $\mu_f(q_n)$ converges weakly to $\mu_f(q)$.\\
	Let $B$ be the open set $\R\backslash\Sigma$. Since the support of $\mu_f(q_n)$ is a subset of the spectrum of $\cH(\xi_n)$, that is, a subset of $\Sigma$, the weak convergence stated above yields
	$$ \mu_{f}(q)(B) \le \liminf_{n\to\infty} \mu_{f}(q_n)(B) = 0\;.$$
	Consequently, the support of $\mu_f(q)$ is included into $\Sigma$. Since the spectrum of $\cH_\eps(\xi)$ is the closure of the union over all $f\in L^2$ of the support of $\mu_f(q)$, this suffices to deduce that $\Sigma_\eps \subset \Sigma$. Proposition \ref{prop:Sigma_eps} shows that $\Sigma_\eps = \R$, and therefore $\Sigma = \R$.
\end{proof}

\bibliographystyle{Martin}
\bibliography{ref}

\begin{thebibliography}{{Ugu}22}
\expandafter\ifx\csname url\endcsname\relax
  \def\url#1{\texttt{#1}}\fi
\expandafter\ifx\csname urlprefix\endcsname\relax\def\urlprefix{URL }\fi
\expandafter\ifx\csname href\endcsname\relax
  \def\href#1#2{#2}\fi
\expandafter\ifx\csname burlalt\endcsname\relax
  \def\burlalt#1#2{\href{#2}{\texttt{#1}}}\fi

\bibitem[AC15]{AC15}
\textsc{R.~{Allez}} and \textsc{K.~{Chouk}}.
\newblock {The continuous Anderson hamiltonian in dimension two}.
\newblock \emph{arXiv e-prints}  arXiv:1511.02718.
\newblock \burlalt{arXiv:1511.02718}{http://arxiv.org/abs/1511.02718}.
\newblock
  \burlalt{doi:10.48550/arXiv.1511.02718}{http://dx.doi.org/10.48550/arXiv.1511.02718}.

\bibitem[And58]{Anderson58}
\textsc{P.~W. Anderson}.
\newblock {Absence of Diffusion in Certain Random Lattices}.
\newblock \emph{Physical Review} \textbf{109}, no.~5, (1958), 1492--1505.
\newblock
  \burlalt{doi:10.1103/PhysRev.109.1492}{http://dx.doi.org/10.1103/PhysRev.109.1492}.

\bibitem[BDM22]{BDM}
\textsc{I.~{Bailleul}}, \textsc{N.~V. {Dang}}, and \textsc{A.~{Mouzard}}.
\newblock {Analysis of the Anderson operator}.
\newblock \emph{arXiv e-prints}  arXiv:2201.04705.
\newblock \burlalt{arXiv:2201.04705}{http://arxiv.org/abs/2201.04705}.

\bibitem[CF18]{CF16}
\textsc{K.~Chouk} and \textsc{P.~K. Friz}.
\newblock Support theorem for a singular {SPDE}: the case of g{PAM}.
\newblock \emph{Ann. Inst. Henri Poincar\'{e} Probab. Stat.} \textbf{54},
  no.~1, (2018), 202--219.
\newblock
  \burlalt{doi:10.1214/16-AIHP800}{http://dx.doi.org/10.1214/16-AIHP800}.

\bibitem[CL90]{CL90}
\textsc{R.~Carmona} and \textsc{J.~Lacroix}.
\newblock \emph{Spectral theory of random {S}chr\"{o}dinger operators}.
\newblock Probability and its Applications. Birkh\"{a}user Boston, Inc.,
  Boston, MA, 1990.

\bibitem[Cv21]{Cv21}
\textsc{K.~Chouk} and \textsc{W.~{van Zuijlen}}.
\newblock Asymptotics of the eigenvalues of the {{Anderson Hamiltonian}} with
  white noise potential in two dimensions.
\newblock \emph{The Annals of Probability} \textbf{49}, no.~4(2021).
\newblock
  \burlalt{doi:10.1214/20-aop1497}{http://dx.doi.org/10.1214/20-aop1497}.

\bibitem[FL74]{FL74}
\textsc{W.~G. Faris} and \textsc{R.~B. Lavine}.
\newblock Commutators and self-adjointness of {{Hamiltonian}} operators.
\newblock \emph{Communications in Mathematical Physics} \textbf{35}, no.~1,
  (1974), 39--48.
\newblock
  \burlalt{doi:10.1007/BF01646453}{http://dx.doi.org/10.1007/BF01646453}.

\bibitem[GIP15]{GIP}
\textsc{M.~Gubinelli}, \textsc{P.~Imkeller}, and \textsc{N.~Perkowski}.
\newblock Paracontrolled distributions and singular {PDE}s.
\newblock \emph{Forum Math. Pi} \textbf{3}, (2015), e6, 75.
\newblock \burlalt{arXiv:1210.2684}{http://arxiv.org/abs/1210.2684}.
\newblock
  \burlalt{doi:10.1017/fmp.2015.2}{http://dx.doi.org/10.1017/fmp.2015.2}.

\bibitem[GUZ20]{GUZ}
\textsc{M.~Gubinelli}, \textsc{B.~Ugurcan}, and \textsc{I.~Zachhuber}.
\newblock Semilinear evolution equations for the {A}nderson {H}amiltonian in
  two and three dimensions.
\newblock \emph{Stoch. Partial Differ. Equ. Anal. Comput.} \textbf{8}, no.~1,
  (2020), 82--149.
\newblock
  \burlalt{doi:10.1007/s40072-019-00143-9}{http://dx.doi.org/10.1007/s40072-019-00143-9}.

\bibitem[Hai14]{Hai14}
\textsc{M.~Hairer}.
\newblock A theory of regularity structures.
\newblock \emph{Inventiones Mathematicae} \textbf{198}, no.~2, (2014),
  269--504.
\newblock
  \burlalt{doi:10.1007/s00222-014-0505-4}{http://dx.doi.org/10.1007/s00222-014-0505-4}.

\bibitem[HL15]{HL15}
\textsc{M.~Hairer} and \textsc{C.~Labb\'{e}}.
\newblock A simple construction of the continuum parabolic {A}nderson model on
  {${\bf R}^2$}.
\newblock \emph{Electron. Commun. Probab.} \textbf{20}, (2015), no. 43, 11.
\newblock
  \burlalt{doi:10.1214/ECP.v20-4038}{http://dx.doi.org/10.1214/ECP.v20-4038}.

\bibitem[HL18]{HL18}
\textsc{M.~Hairer} and \textsc{C.~Labb{\'e}}.
\newblock Multiplicative stochastic heat equations on the whole space.
\newblock \emph{Journal of the European Mathematical Society} \textbf{20},
  no.~4, (2018), 1005--1054.
\newblock \burlalt{doi:10.4171/JEMS/781}{http://dx.doi.org/10.4171/JEMS/781}.

\bibitem[HL23]{HL22}
\textsc{Y.-S. Hsu} and \textsc{C.~Labb\'{e}}.
\newblock Asymptotic of the smallest eigenvalues of the continuous {A}nderson
  {H}amiltonian in {$d\le 3$}.
\newblock \emph{Stoch. Partial Differ. Equ. Anal. Comput.} \textbf{11}, no.~3,
  (2023), 1089--1122.
\newblock
  \burlalt{doi:10.1007/s40072-022-00252-y}{http://dx.doi.org/10.1007/s40072-022-00252-y}.

\bibitem[HP15]{HP15}
\textsc{M.~Hairer} and \textsc{{\'E}.~Pardoux}.
\newblock A {{Wong-Zakai}} theorem for stochastic {{PDEs}}.
\newblock \emph{Journal of the Mathematical Society of Japan} \textbf{67},
  no.~4, (2015), 1551--1604.
\newblock
  \burlalt{doi:10.2969/jmsj/06741551}{http://dx.doi.org/10.2969/jmsj/06741551}.

\bibitem[HQ18]{HQ18}
\textsc{M.~Hairer} and \textsc{J.~Quastel}.
\newblock A class of growth models rescaling to {{KPZ}}.
\newblock \emph{Forum of Mathematics, Pi} \textbf{6}, (2018), e3.
\newblock
  \burlalt{doi:10.1017/fmp.2018.2}{http://dx.doi.org/10.1017/fmp.2018.2}.

\bibitem[HS22]{HS21}
\textsc{M.~Hairer} and \textsc{P.~Sch\"{o}nbauer}.
\newblock The support of singular stochastic partial differential equations.
\newblock \emph{Forum Math. Pi} \textbf{10}, (2022), Paper No. e1, 127.
\newblock
  \burlalt{doi:10.1017/fmp.2021.18}{http://dx.doi.org/10.1017/fmp.2021.18}.

\bibitem[HW14]{HW14}
\textsc{M.~Hairer} and \textsc{H.~Weber}.
\newblock Large deviations for white-noise driven, nonlinear stochastic pdes in
  two and three dimensions.
\newblock \emph{Annales de la facult{\'e} des sciences de Toulouse
  Math{\'e}matiques} \textbf{24}, no.~1, (2014), 55--92.
\newblock \burlalt{doi:10.5802/afst.1442}{http://dx.doi.org/10.5802/afst.1442}.

\bibitem[Kat72]{Kat72}
\textsc{T.~Kato}.
\newblock Schr{\"o}dinger operators with singular potentials.
\newblock \emph{Israel Journal of Mathematics} \textbf{13}, no. 1-2, (1972),
  135--148.
\newblock
  \burlalt{doi:10.1007/BF02760233}{http://dx.doi.org/10.1007/BF02760233}.

\bibitem[KL81]{KL81}
\textsc{A.~Klein} and \textsc{L.~J. Landau}.
\newblock Construction of a unique self-adjoint generator for a symmetric local
  semigroup.
\newblock \emph{Journal of Functional Analysis} \textbf{44}, no.~2, (1981),
  121--137.
\newblock
  \burlalt{doi:10.1016/0022-1236(81)90007-0}{http://dx.doi.org/10.1016/0022-1236(81)90007-0}.

\bibitem[Kot85]{Kot85}
\textsc{S.~Kotani}.
\newblock Support theorems for random {S}chr\"{o}dinger operators.
\newblock \emph{Comm. Math. Phys.} \textbf{97}, no.~3, (1985), 443--452.

\bibitem[Lab19]{Lab19}
\textsc{C.~Labb\'{e}}.
\newblock The continuous {A}nderson {H}amiltonian in {$d\leq 3$}.
\newblock \emph{J. Funct. Anal.} \textbf{277}, no.~9, (2019), 3187--3235.
\newblock
  \burlalt{doi:10.1016/j.jfa.2019.05.027}{http://dx.doi.org/10.1016/j.jfa.2019.05.027}.

\bibitem[Mat22]{Mat21}
\textsc{T.~Matsuda}.
\newblock Integrated density of states of the {A}nderson {H}amiltonian with
  two-dimensional white noise.
\newblock \emph{Stochastic Process. Appl.} \textbf{153}, (2022), 91--127.
\newblock
  \burlalt{doi:10.1016/j.spa.2022.07.007}{http://dx.doi.org/10.1016/j.spa.2022.07.007}.

\bibitem[Mou22]{Mou21}
\textsc{A.~Mouzard}.
\newblock Weyl law for the {A}nderson {H}amiltonian on a two-dimensional
  manifold.
\newblock \emph{Ann. Inst. Henri Poincar\'{e} Probab. Stat.} \textbf{58},
  no.~3, (2022), 1385--1425.
\newblock
  \burlalt{doi:10.1214/21-aihp1216}{http://dx.doi.org/10.1214/21-aihp1216}.

\bibitem[Mv22]{MvZ}
\textsc{T.~{Matsuda}} and \textsc{W.~{van Zuijlen}}.
\newblock {Anderson Hamiltonians with singular potentials}.
\newblock \emph{arXiv e-prints}  arXiv:2211.01199.
\newblock \burlalt{arXiv:2211.01199}{http://arxiv.org/abs/2211.01199}.
\newblock
  \burlalt{doi:10.48550/arXiv.2211.01199}{http://dx.doi.org/10.48550/arXiv.2211.01199}.

\bibitem[Tes14]{Tes14}
\textsc{G.~Teschl}.
\newblock \emph{Mathematical methods in quantum mechanics}, vol. 157 of
  \emph{Graduate Studies in Mathematics}.
\newblock American Mathematical Society, Providence, RI, second ed., 2014.
\newblock With applications to Schr\"{o}dinger operators.

\bibitem[Uek25]{Uek23}
\textsc{N.~Ueki}.
\newblock A definition of self-adjoint operators derived from the
  {S}chr\"{o}dinger operator with the white noise potential on the plane.
\newblock \emph{Stochastic Process. Appl.} \textbf{186}, (2025), Paper No.
  104642, 30.
\newblock
  \burlalt{doi:10.1016/j.spa.2025.104642}{http://dx.doi.org/10.1016/j.spa.2025.104642}.

\bibitem[{Ugu}22]{Ugu22}
\textsc{B.~E. {Ugurcan}}.
\newblock {Anderson Hamiltonian and associated Nonlinear Stochastic Wave and
  Schr{\"o}dinger equations in the full space}.
\newblock \emph{arXiv e-prints}  arXiv:2208.09352.
\newblock \burlalt{arXiv:2208.09352}{http://arxiv.org/abs/2208.09352}.
\newblock
  \burlalt{doi:10.48550/arXiv.2208.09352}{http://dx.doi.org/10.48550/arXiv.2208.09352}.

\bibitem[Wei97]{Wei97}
\textsc{J.~Weidmann}.
\newblock Strong operator convergence and spectral theory of ordinary
  differential operators.
\newblock \emph{Univ. Iagel. Acta Math.} , no.~34, (1997), 153--163.

\bibitem[Wid34]{Widder}
\textsc{D.~V. Widder}.
\newblock Necessary and sufficient conditions for the representation of a
  function by a doubly infinite {L}aplace integral.
\newblock \emph{Bull. Amer. Math. Soc.} \textbf{40}, no.~4, (1934), 321--326.
\newblock
  \burlalt{doi:10.1090/S0002-9904-1934-05862-2}{http://dx.doi.org/10.1090/S0002-9904-1934-05862-2}.

\end{thebibliography}

\end{document}